\documentclass[10pt]{amsart}
\usepackage[all,cmtip]{xy}
\usepackage{amsmath, amsthm, amssymb}
\usepackage{eucal, mathrsfs, yfonts}
\usepackage{amsfonts}
\usepackage{array}
\usepackage{hyperref}
\usepackage{graphicx}
\usepackage{caption}
\usepackage{subcaption}

\newtheorem{thm}{Theorem}[section]

\newtheorem{lem}[thm]{Lemma}

\newtheorem{cor}[thm]{Corollary}
\theoremstyle{definition}

\theoremstyle{remark}
\newtheorem{rem}[thm]{Remark}
\theoremstyle{Conjecture}
\newtheorem{conj}[thm]{Conjecture}
\numberwithin{equation}{section}

\newcommand{\R}{\mathbb{R}}

\begin{document}

\title{Energy measures of harmonic functions on the Sierpi\'nski Gasket}
\author[Renee Bell]{Renee Bell$^1$}
\address{Renee Bell\\
         Department of Mathematics\\
         University of California, Berkeley\\
         Evans Hall\\
         Berkeley, CA 94720-3840\\
         U.S.A}
\email{rbell916@berkeley.edu}
\thanks{$^1$Research supported by the National Science Foundation through the Research Experiences for Undergraduates(REU) Program at Cornell.}

\author[Ching-Wei Ho]{Ching-Wei Ho}
\address{Ching-Wei Ho\\
         Department of Mathematics\\
         Lady Shaw Building\\
         The Chinese University of Hong Kong\\
         Shatin, Hong Kong}
\email{s0962318@cuhk.edu.hk}

\author[Robert S. Strichartz]{Robert S. Strichartz$^2$}
\address{Robert S. Strichartz\\
         Department of Mathematics\\
         Malott Hall\\
         Cornell University\\
         Ithaca, NY 14853\\
         U.S.A}
\email{str@math.cornell.edu}
\thanks{$^2$Research supported by the National Science Foundation grant DMS-1162045}



\maketitle

\begin{abstract}
We study energy measures on SG based on harmonic functions. We characterize the positive energy measures through studying the bounds of Radon-Nikodym derivatives with respect to the Kusuoka measure. We prove a limited continuity of the derivative on the graph $V_*$
and express the average value of the derivative on a whole cell as a weighted average of the values on the boundary vertices.
We also prove some characterizations and properties of the weights.
\end{abstract}

\section{introduction}
In the development of analysis on fractals, there are three basic concepts: energy, measure and Laplacian. For $SG$, which may be regarded as the ``poster child" for the theory, Kigami has developed an elegant construction of a self-similar Laplacian based on a self-similar energy $\mathcal{E}$ and self-similar measure $\mu$ via an analog of integration by parts. This theory is briefly sketched below, and is described in complete detail in the book \cite{Bob}. This standard Laplacian $\Delta$ is not the only Laplacian on $SG$ that has been widely studied. A competing energy Laplacian $\Delta_\nu$, based on the same energy but using the Kusuoka measure $\nu$ in place of $\mu$, has received a lot of attention in recent years(\cite{Kajino, Kajino2, Kigami, Teplyaev, Teplyaev2}). It has two striking advantages over $\Delta$, namely the existence of a Leibniz-type product formula and the fact that the associated heat kernel satisfies Gaussian type estimates. The main drawback so far has been that the Kusuoka measure is not algorithmically transparent. The purpose of this paper is to eliminate this drawback. We will provide algorithms that describe the nature of this measure and more generally the family of energy measures $\nu_h$ based on harmonic functions $h$, and the Radon-Nikodym derivatives $\frac{d\nu_h}{d\nu}$. Some of the results of this paper, especially Theorem \ref{thm:LimitedContinuity}, may also be derived from \cite{Kigami}.

The Sierpi\'nski gasket, or simply SG, is the unique nonempty compact set satisfying
$$SG=\bigcup_{i=0}^2 F_i SG$$
where $F_i=\frac{1}{2}(x+q_i)$ and $\{q_i\}_{i=0}^2$ are vertices of an equilateral triangle in counter-clockwise direction. $\{q_i\}_{i=1}^2$ are called the boundary points of SG. If $w=(w_1,\cdots, w_n)$ is a finite word, we define the mapping $F_w = F_{w_1}\circ \cdots \circ F_{w_n}$ where $|w|$ is defined to be the length of the word. Define $V_0=\{q_i\}$, $V_n = \bigcup F_i \left(V_{n-1}\right)$ and $V_*=\bigcup_{n=1}^\infty V_n$. For the boundary of a cell $C=F_w SG$ , denoted by $\partial C$, we mean the three vertices $\{F_w(q_i)\}_{i=0}^2$ of the cell.\\

For functions $u$, $v$ defined on SG, we define the energy on level $m$ by
\begin{equation}\
\label{eq:EnergyOnLvlm}
\mathcal{E}_m(u,v)=\left(\frac{5}{3}\right)^m\sum_{|w|=m}\sum_{i<j}(u\circ F_w(q_i)-u\circ F_w(q_j))(v\circ F_w(q_i)-v\circ F_w(q_j)).
\end{equation}
and energy $\mathcal{E}(u,v)=\lim\limits_{m\to\infty}\mathcal{E}_m(u,v)$ provided that $\mathcal{E}(u,u)$ and $\mathcal{E}(v,v)$ are finite. For each cell $F_w SG$, we define the energy of $u$, $v$ on this cell by
$$\mathcal{E}_{F_w SG}(u,v)=\left(\frac{5}{3}\right)^{|w|}\mathcal{E}(u\circ F_w, v\circ F_w).$$
Observe that the energy is non-negative if $u=v$, but in general it is not. In the case $u=v$, we write $\mathcal{E}(u)$ instead of $\mathcal{E}(u,u)$. We define a function $h$ to be harmonic if it minimizes the energy from level $m$ to level $m+1$, as defined in \eqref{eq:EnergyOnLvlm}. A simple computation shows that this minimization occurs when $\mathcal{E}_m(h)=\mathcal{E}_{m+1}(h)$ at which to each cell $F_w SG$, $h(F_{w0}(q_1))=\frac{2}{5}h(F_w(q_0))+\frac{2}{5}h(F_w(q_1))+\frac{1}{5}h(F_w(q_2))$. Equivalently, a function $h$ on SG is harmonic if the sequence, taking $u=h=v$ in \eqref{eq:EnergyOnLvlm}, is a constant sequence, i.e. constantly $\sum_{i<j}(h(q_i)-h(q_j))^2$. Given three values on the boundary, in order to have a constant sequence in \eqref{eq:EnergyOnLvlm}, all the other values are completely determined. Consequently, all the harmonic functions form a three-dimensional space and hence any values on the boundary of a cell can extend to a unique harmonic function on SG. A harmonic function $h$ is said to be symmetric(resp. skew-symmetric) if $h(q_i)=h(q_{i+1})$(resp. $h(q_i)=-h(q_{i+1})$) for some $i$. Harmonic functions characterized by $h_i(q_j)=\delta_{ij}$ are three symmetric functions and the notation $h_i$ will be used throughout this paper. On the space $\mathcal{H}$ of harmonic functions modulo constants, $(\mathcal{H},\mathcal{E}$) is a two-dimensional inner product space. A harmonic function $h$ is said to be symmetric (resp. skew-symmetric) in $\mathcal{H}$ if there exists a constant $c$ and a symmetric (resp. skew-symmetric) harmonic function $f$ such that $h=c+f$. Explicitly, for a symmetric harmonic function $h$, we mean $h=ah_i+c$ for some $a,c\in \R$ and some $h_i$; and for a skew-symmetric harmonic function $h$, we mean $a(h_i-h_j)+c$ for some $a>0$ and $i\not=j$. To each harmonic $h$, we write $h^\perp$ the harmonic function having the same energy as $h$ and orthogonal to $h$ unless we specify $h^\perp$ ``orthonormal to $h$" in which case we mean having energy $1$ and orthogonal to $h$.\\

The standard measure $\mu$ is a self-similar measure characterized by $\mu(F_w F_i SG)=\frac{1}{3}\mu(F_w SG)$ for every cell $C$. The standard Laplacian $\Delta u$ of $u$ is a continuous function satisfying
\begin{equation}
\label{eq:StdLap}
-\mathcal{E}(u,v)=\int \Delta u v d\mu
\end{equation}
for every $v$ vanishing at the boundary with $\mathcal{E}(v)<\infty$. Harmonic functions on SG are precisely those whose Laplacian is zero. However, for a function $u$ whenever $\Delta u$ is defined as a function, $\Delta (u^2)$ is not defined as a function. We can view $\Delta u$ as a measure $m = \Delta u d\mu$ and $\Delta u$ exists as a function if and only if this measure is absolutely continuous with respect to the standard measure. The carr\'e du champs formula given by, for all $f, u, v$ of finite energy, $\int_{SG} f d\nu_{u,v}=\frac{1}{2}\mathcal{E}(fu,v)+\frac{1}{2}\mathcal{E}(u,fv)-\frac{1}{2}\mathcal{E}(f,uv)$ shows that $$\mathcal{E}(u^2,v)=-2\int (u\Delta u)vd\mu - 2\int vd\nu_u$$
for all $u, v$ with $\Delta u$ as a function. Thus if $\Delta u$ exists as a function and if we first view $\Delta(u^2)$ as a measure,
$$\Delta(u^2)=2u\Delta u d\mu+2d\nu_u$$
then $\Delta(u^2)$ would exist as a function if this measure were absolutely continuous with respect to the standard measure $\mu$, but this is almost never the case \cite{Bassat}.

Suppose that $u$, $v$ are functions of finite energy. We can define a signed measure $\nu_{u,v}$ by, on each cell,
$$\nu_{u,v}(C)=\mathcal{E}_C(u,v).$$
As usual, if it happens that $u=v$, we write $\nu_u$ instead. In this paper, we will restrict the attention to both $u$, $v$ being harmonic. The energy measures of harmonic functions form a three-dimensional space. The Kusuoka measure $\nu$ is defined to be $\nu_0+\nu_1+\nu_2$. An easy computation shows that for any harmonic function $h$, we have $\nu_h+\nu_{h^\perp}=c\nu$ where $c=\frac{1}{3}\mathcal{E}(h)$. Now, we define the energy Laplacian $\Delta_\nu u$ of $u$ by, for every finite energy $v$ vanishing on the boundary,
\begin{equation}
\label{eq:EnLap}
-\mathcal{E}(u,v)=\int (\Delta_\nu u) v d\nu.
\end{equation}
By Theorem 5.3.1 of \cite{Bob}, every energy measure of harmonic functions is absolutely continuous with respect to the Kusuoka measure $\nu$ and it makes sense to consider the Radon-Nikodym derivatives. For the energy measures of the symmetric harmonic functions, we will denote these by $\nu_i$ instead of $\nu_{h_i}$. $\{\nu_i\}$ forms a basis of the space of all energy measures of harmonic functions. Since $h_0+h_1+h_2=1$, we have $\nu_{h_0,h_j}+\nu_{h_1,h_j}+\nu_{h_2,h_j}=\nu_{1,h_j}=0$ for $j=0,1,2$. Since $\nu_{f,g}=\nu_{g,f}$ for all $f$, $g$, this gives us three equations that we solve to obtain
\begin{equation}
\label{eq:EnMeaRep}
\begin{split}
\nu_{h_0,h_1}&=\frac{1}{2}(-\nu_0-\nu_1+\nu_2);\\
\nu_{h_0,h_2}&=\frac{1}{2}(-\nu_0+\nu_1-\nu_2);\\
\nu_{h_1,h_2}&=\frac{1}{2}(\hspace{8pt}\nu_0-\nu_1-\nu_2).
\end{split}
\end{equation}
We will keep using the notation $\nu_{f,g}$ for any energy measure of harmonic functions without mentioning $f$ and $g$.

By \cite{EMatrix}, we have that there exist matrices $E_i$ such that for every cell $C$,
\begin{equation*}
\left(\begin{array}{c}
\nu_{f,g}(F_{i0}C)\\
\nu_{f,g}(F_{i1}C)\\
\nu_{f,g}(F_{i2}C)
\end{array}\right)
= E_i
\left(\begin{array}{c}
\nu_{f,g}(F_{0}C)\\
\nu_{f,g}(F_{1}C)\\
\nu_{f,g}(F_{2}C)
\end{array}\right).
\end{equation*}
The $E_i$'s are all diagonalizable, so it is easy to calculate $E_i^m$. Explicitly, after diagonalization, we have\\
\begin{equation}
\label{eq:DiagE}
\begin{split}
E_0 &= \left(\begin{array}{ccc}
\frac{1}{7}&3&0\\[1pt]  1&1&-1\\[1pt]  1&1&1
\end{array}\right)
\left(\begin{array}{ccc}
\frac{1}{15}&0&0\\[1pt]  0&\frac{3}{5}&0\\[1pt]  0&0&\frac{1}{5}
\end{array}\right)
\left(\begin{array}{ccc}
\frac{-7}{20}&\frac{21}{40}&\frac{21}{40}\\[1pt]
\frac{7}{20}&\frac{-1}{40}&\frac{-1}{40}\\[1pt]
0&-\frac{1}{2}&\frac{1}{2}
\end{array}\right);\\
E_1 &= \left(\begin{array}{ccc}
1&1&-1\\[1pt]  \frac{1}{7}&3&0\\[1pt]  1&1&1
\end{array}\right)
\left(\begin{array}{ccc}
\frac{1}{15}&0&0\\[1pt]  0&\frac{3}{5}&0\\[1pt]  0&0&\frac{1}{5}
\end{array}\right)
\left(\begin{array}{ccc}
\frac{21}{40}&\frac{-7}{20}&\frac{21}{40}\\[1pt]
\frac{-1}{40}&\frac{7}{20}&\frac{-1}{40}\\[1pt]
-\frac{1}{2}&0&\frac{1}{2}
\end{array}\right);\\
E_2 &= 
\left(\begin{array}{ccc}
1&1&-1\\[1pt]  1&1&1\\[1pt]  \frac{1}{7}&3&0
\end{array}\right)
\left(\begin{array}{ccc}
\frac{1}{15}&0&0\\[1pt]  0&\frac{3}{5}&0\\[1pt]  0&0&\frac{1}{5}
\end{array}\right)
\left(\begin{array}{ccc}
\frac{21}{40}&\frac{21}{40}&\frac{-7}{20}\\[1pt]
\frac{-1}{40}&\frac{-1}{40}&\frac{7}{20}\\[1pt]
-\frac{1}{2}&\frac{1}{2}&0
\end{array}\right).\\
\end{split}
\end{equation}
It is easy to calculate that the limits
of quotients which are used to compute derivatives are as follows:
$$\lim_m \frac{\left(\begin{array}{ccc}
1&1&1
\end{array}\right)E_0^m E_1\left(
\begin{array}{c}
x\\y\\z
\end{array}\right)}{
\left(\begin{array}{ccc}
1&1&1
\end{array}\right)E_0^m E_1
\left(\begin{array}{c}
\alpha\\ \beta\\ \gamma
\end{array}\right)
}
= \lim_m \frac{\left(\begin{array}{ccc}
1&1&1
\end{array}\right)E_1^m E_0\left(
\begin{array}{c}
x\\y\\z
\end{array}\right)}{
\left(\begin{array}{ccc}
1&1&1
\end{array}\right)E_1^m E_0
\left(\begin{array}{c}
\alpha\\ \beta\\ \gamma
\end{array}\right)
}$$
where
$$\left(\begin{array}{c} x\\y\\z\end{array}\right)=\left(\begin{array}{c}\nu_{f,g}(F_0 C)\\ \nu_{f,g}(F_1 C)\\ \nu_{f,g}(F_2 C)\end{array}\right), \left(\begin{array}{c} \alpha\\ \beta\\ \gamma\end{array}\right)=\left(\begin{array}{c}\nu(F_0 C)\\ \nu(F_1 C)\\ \nu(F_2 C)\end{array}\right).$$
For each infinite word $w$, let $w_m$ be the word consisting of the first $m$ letters of $w$; then $\lim\limits_{m\to \infty} E_{w_m} = 0$. It follows that every energy measure is continuous.

It is natural to ask whether the energy Laplacian behaves in a similar manner to the standard Laplacian; in Section 2, we express the ``self-similarity" of the energy Laplacian. In Section 3, we will study the decay rates of energy measure from a cell to its subcells. Through studying this, we give sharp bounds for the Radon-Nikodym derivatives for each $\nu_h$. Being in general a signed measure, it is natural to ask when it is a (positive) measure; in Section 4, we will characterize all the positive energy measures. It is well-known that the derivative is not continuous; nevertheless, in Section 5, we show that a certain restriction of the derivative is continuous. In Section 6, we will express, on each cell, the average value of derivatives by a weighted average of values on vertices of the cell. How the energy distributes is mysterious; in Section 7, we will provide some graphs and hypothesis about this question.

\section{Self-similarity of energy Laplacian}
In this section we discuss the self-similarity identities for energy measures and the energy Laplacian. We will see that the Radon-Nikodym derivatives $R_i = \frac{d\nu_i}{d\nu}$ play a crucial role in these identities. Individually, none of the $\nu_i$ is self-similar; nevertheless, together they form a self-similar family, in the sense of Mauldin-Williams \cite{MW}. We begin with considering the symmetric function $h_0$. We have the following equalities,
\begin{equation}
\label{eq:h_0Decompose}
\begin{split}
&h_0 \circ F_0 = h_0 + \frac{2}{5}h_1 + \frac{2}{5}h_2 = \frac{2}{5} + \frac{3}{5}h_0;\\
&h_0 \circ F_1 = \frac{2}{5}h_0 + \frac{1}{5}h_2;\\
&h_0 \circ F_2 = \frac{2}{5}h_0 + \frac{1}{5}h_1.
\end{split}
\end{equation}
Next we establish the relation of $\nu_i$'s from cells to subcells.
\begin{thm}
\label{thm:MMatrices}
Let
$$M_0 = \frac{1}{15}\left(
\begin{array}{ccc}
9&0&0\\
2&2&-1\\
2&-1&2
\end{array}
\right),
M_1 = \frac{1}{15}\left(
\begin{array}{ccc}
2&2&-1\\
0&9&0\\
-1&2&2
\end{array}
\right),
M_2 = \frac{1}{15}\left(
\begin{array}{ccc}
2&-1&2\\
-1&2&2\\
0&0&9
\end{array}
\right),
$$
then
$$
\left(
\begin{array}{c}
\nu_0(F_i C)\\ \nu_1(F_i C) \\ \nu_2(F_i C)
\end{array}
\right) =
M_i
\left(
\begin{array}{c}
\nu_0(C)\\ \nu_1(C) \\ \nu_2(C)
\end{array}
\right)
$$
for every cell C. In other words,
$$
\left(
\begin{array}{c}
\nu_0\\ \nu_1\\ \nu_2
\end{array}
\right) =
M_0
\left(
\begin{array}{c}
\nu_0\\ \nu_1 \\ \nu_2
\end{array}
\right)\circ F_0^{-1} +
M_1
\left(
\begin{array}{c}
\nu_0\\ \nu_1 \\ \nu_2
\end{array}
\right)\circ F_1^{-1} +
M_2
\left(
\begin{array}{c}
\nu_0\\ \nu_1 \\ \nu_2
\end{array}
\right)\circ F_2^{-1}.
$$
\end{thm}
\begin{proof}
Let $f$ be a contintuous function on SG, by \eqref{eq:EnMeaRep}, \eqref{eq:h_0Decompose},
\begin{equation*}
\begin{split}
\int_{F_0 SG}{f d\nu_0}
= &\frac{5}{3} \int_{SG}f\circ F_0 \,d\nu_{\frac{2}{5} + \frac{3}{5}h_0}\\
= &\frac{5}{3}\left(\frac{3}{5}\right)^2\int_{SG}f\circ F_0\,d\nu_0\\
= &\frac{3}{5}\int_{SG}f\circ F_0\,d\nu_0.
\end{split}
\end{equation*}
On the other hand,
\begin{equation*}
\begin{split}
&\int_{F_1 SG}f\, d\nu_0 \\
= &\frac{5}{3}\int_{SG}f\circ F_1 \,d\nu_{\frac{2}{5}h_0+\frac{1}{5}h_2}\\
=&\frac{2}{15}\int_{SG}{f \circ F_1 \,d\nu_0} + \frac{2}{15}\int_{SG}{f\circ F_1 \,d \nu_1} - \frac{1}{15}\int_{SG}{f\circ F_1 \,d\nu_2}
\end{split}
\end{equation*}
using \eqref{eq:EnMeaRep} again. Similarly,
\begin{equation*}
\begin{split}
&\int_{F_2 SG}f \,d\nu_0 \\
= &\frac{5}{3}\int_{SG}f\circ F_1 \,d\nu_{\frac{2}{5}h_0+\frac{1}{5}h_2}\\
=&\frac{2}{15}\int_{SG}{f \circ F_2 \,d\nu_0} - \frac{1}{15}\int_{SG}{f\circ F_2 \,d \nu_1} + \frac{2}{15}\int_{SG}{f\circ F_2 \,d\nu_2}.\\
\end{split}
\end{equation*}
Since this is true for arbitrary continuous function $f$, we have the following identity
\begin{equation*}
\begin{split}
\nu_0 = &\hspace{12
pt} \frac{9}{15} \nu_0 \circ F_0^{-1}\\
&+ \frac{2}{15}\nu_0 \circ F_1^{-1} + \frac{2}{15}\nu_1 \circ F_1^{-1} - \frac{1}{15}\nu_2 \circ F_1^{-1}\\
&+ \frac{2}{15}\nu_0 \circ F_2^{-1} - \frac{1}{15}\nu_1 \circ F_2^{-1} + \frac{2}{15}\nu_2 \circ F_2^{-1}.\\
\end{split}
\end{equation*}
By symmetry, we get similar expressions for $\nu_1$ and $\nu_2$, and we have the relation
\begin{equation*}
\left(
\begin{array}{c}
\nu_0\\ \nu_1 \\ \nu_2
\end{array}
\right)
=
M_0 \left(
\begin{array}{c}
\nu_0\\ \nu_1 \\ \nu_2
\end{array}
\right) \circ F_0^{-1} +
M_1 \left(
\begin{array}{c}
\nu_0\\ \nu_1 \\ \nu_2
\end{array}
\right) \circ F_1^{-1} +
M_2 \left(
\begin{array}{c}
\nu_0\\ \nu_1 \\ \nu_2
\end{array}
\right) \circ F_2^{-1}.
\end{equation*}
\end{proof}
\begin{cor}
The Kusuoka measure satisfies the variable weight self-similar identity
\begin{equation}
\label{eq:nuSelfSimilar}
\nu=\sum_{i=0}^2 \left(\left(\frac{1}{15}+\frac{12}{15}R_i\right)\nu \right)\circ F_i^{-1}.
\end{equation}
\end{cor}
\begin{proof}
If we consider $\nu = \left( \begin{array}{ccc} 1&1&1 \end{array}\right)\left( \begin{array}{c} \nu_0\\ \nu_1\\ \nu_2 \end{array}\right)$, where $R_j = \frac{d\nu_j}{d\nu}$, the Radon-Nikodym Derivative of $\nu_j$ w.r.t. $\nu$, we get
\begin{equation}
\begin{split}
\nu = &\frac{12}{15}\nu_0 \circ F_0^{-1} + \frac{1}{15}\nu \circ F_0^{-1}+\frac{12}{15}\nu_1 \circ F_1^{-1} + \frac{1}{15}\nu \circ F_1^{-1}+\frac{12}{15}\nu_2 \circ F_2^{-1} + \frac{1}{15}\nu \circ F_2^{-1}\\
=&\sum_{i=0}^2\left(\left(\frac{1}{15}+\frac{12}{15}R_i\right)\nu \right)\circ F_i^{-1}.
\end{split}
\end{equation}
\end{proof}
We now give the ``self-similarity" of the energy Laplacian $\Delta_\nu$.\\
\begin{thm}
Let $Q_j = \frac{1}{15} + \frac{12}{25}R_j$. Then the energy Laplacian $\Delta_\nu$ satisfies
$$
\Delta_\nu (u \circ F_j) = Q_j (\Delta_\nu u) \circ F_j
$$
\end{thm}
\begin{proof}
By the definition of the energy Laplacian and by \eqref{eq:nuSelfSimilar},
\begin{equation*}
\begin{split}
-\mathcal{E}(u,v) &= \int_{SG}{(\Delta_{\nu}u)v d\nu}\\
&=\sum_{j=0}^2{\int_{SG}{(\frac{1}{15}+\frac{12}{15}R_j)(\Delta_{\nu}u)\circ F_j\hspace{6pt} v \circ F_j}\hspace{6pt}d\nu}.
\end{split}
\end{equation*}
On the other hand, the self-similarity of $\mathcal{E}$ also gives
$$
-\mathcal{E}(u,v) = -\frac{5}{3}\sum_{j=0}^2{\mathcal{E}(u\circ F_j, v\circ F_j)}.
$$
Together we have
$$
\int_{SG}{\sum_{j=0}^2{(\frac{1}{15}+\frac{2}{15}R_j)(\Delta_\nu u)\circ F_j\hspace{6pt}v\circ F_j\hspace{6pt}d\nu}}=\frac{5}{3}\sum_{j=0}^2{\int_{SG}{\Delta_\nu(u\circ F_j)\hspace{6pt} v\circ F_j\hspace{6pt} d\nu}}.
$$
Since $v$ is arbitary, and so $v\circ F_j$ is arbitary, we must have the self-similarity of $\Delta_\nu$
$$
\frac{3}{5}Q_j(\Delta_\nu u)\circ F_j = \Delta_\nu (u\circ F_j).
$$
\end{proof}
\begin{cor}
Let $w = (w_1,...,w_m)$ be a finite word of length $m$ and $F_w = F_{w_1}\circ....\circ F_{w_m}$. Define
$$
Q_w = Q_{w_m} \cdot (Q_{w_{m-1}}\circ F_{w_m}) \cdot (Q_{w_{m-2}}\circ F_{w_{m-1}}\circ F_{w_m})\cdots (Q_{w_1}\circ F_{w_2}\circ \cdots \circ F_{w_m})
$$
Then
$$
\Delta_\nu (u\circ F_w) = Q_w (\Delta_\nu u) \circ F_w.
$$
\end{cor}
\begin{proof}
Use the theorem and iterate.
\end{proof}

To compute $\Delta_\nu u$ on the cell $F_w SG$, zoom in by looking at $u \circ F_w$ as a function on $SG$, compute its Laplacian $\Delta_\nu(u\circ F_w)$ and multiply by $\frac{1}{Q_w}$,
$$
(\Delta_\nu u)\circ F_w = \frac{1}{Q_w} \Delta_\nu(u\circ F_w),
$$
and then zoom out via $F_w^{-1}$
$$
\Delta_\nu u |_{F_w SG} = \left( \frac{1}{Q_w}\Delta_\nu(u\circ F_w)\right)\circ F_w^{-1} = \frac{1}{Q_w\circ F_w^{-1}}\left(\Delta_\nu(u\circ F_w)\right)\circ F_w^{-1}.
$$

\section{Bounds of derivatives}
By the definition of energy measures, we see that every energy measure is continuous; by the self similarity of SG, we see that it makes sense to consider the measure from cells to subcells. It is natural to study how the measure varies for a sequence of cells converging to a point.
\begin{lem}
\label{lemma3.1}
For a harmonic function $h$, if we let $x=\nu_h(F_w F_0 SG)$, $y=\nu_h(F_w F_1 SG)$, $z=\nu_h(F_w F_2 SG)$, then $14x-y-z\geq 0$ with equality if and only if $h$ is skew symmetric (with respect to $F_w q_0$) on the cell $F_w SG$.
\end{lem}
\begin{proof}
WLOG, let $h(F_w q_0)=1$, $h(F_w q_1)=a$, $h(F_w q_2)=0$. Then, writing $m=|w|$ and a direct computation on the
energies on each cell,
\begin{equation*}
\begin{split}
x&=\left(\frac{5}{3}\right)^{m+1}\left( \frac{6}{25}\right)\left(a^2-3a+3 \right); \\
y&=\left(\frac{5}{3}\right)^{m+1}\left( \frac{6}{25}\right)\left(3a^2-3a+1 \right); \\
z&= \left(\frac{5}{3}\right)^{m+1}\left( \frac{6}{25}\right)\left(a^2+a+1 \right)
\end{split}
\end{equation*}
so
$14x-y-z= 10\left( \frac{5}{3} \right)^m \left( \frac{6}{25} \right) \left( a^2-4a+4 \right)$
and
$a^2-4a+4 = (a-2)^2 \geq 0$. Equality holds iff $a=2$, in which case $h$ is skew symmetric.
\end{proof}

\begin{lem} Let $w$ be a word. For
$$
\left(
\begin{array}{c}
x\\
y\\
z
\end{array}
\right)=
E_w
\left(
\begin{array}{c}
2\\
2\\
2
\end{array}
\right),
$$
we have
$$
14x-y-z>0.
$$
\end{lem}
\begin{proof}
Since $\nu(SG)=6$, we have by symmetry $\nu(F_iSG)=2$ for $i=0,1,2$. Thus
$$\left(\begin{array}{c} x\\y\\z\end{array}\right)=
\left(\begin{array}{c} \nu(F_wF_0SG)\\ \nu(F_wF_1SG)\\ \nu(F_wF_2SG)\\\end{array}\right).$$
Now Lemma \ref{lemma3.1} implies $14\nu_i(F_wF_0SG)-\nu_i(F_wF_1SG)-\nu_i(F_wF_2SG)\geq0$ for all $i$, so adding these up yields $14x-y-z\geq 0$. To complete the proof, we show that $14\nu_i(F_w F_0SG)-\nu_i(F_w F_1SG)-\nu_i(F_w F_2SG) > 0$ for some $i$. Suppose, on the contrary, that each $h_i$ is skew symmetric with respect to $F_w q_0$ on the cell $F_w SG$. Then there exist $c_i, a_i$ such that $\left(h_i \circ F_w (q_0), h_i \circ F_w (q_1), h_i \circ F_w (q_2)\right) = \left(c_i, c_i + a_i, c_i-a_i\right) = c_i(1,1,1) + a_i(0,1,-1)$, but this shows $h_i$ lie in a two-dimensional subspace, contradicting ${h_i}$ being linearly independent.
\end{proof}

\begin{lem}
\label{lemma3.3}
For the sequence $\{F_w F_i^m SG\}_m$, which converges to the point $F_w(q_i)$, we have that $\nu(F_w F_i^m SG)=\Theta \left( \left( \frac{3}{5} \right)^{m} \right)$, (here $A_m=\Theta(B_m)$ means $A_m=O(B_m)$ and $B_m=O(A_m)$).
\end{lem}
\begin{proof}
We prove this statement for $i=0$; the other cases are similar.

We know that
\begin{equation*}
\begin{split}
\nu(F_w F_0^m SG)&=\nu_0(F_w F_0^m SG)+\nu_1(F_w F_0^m SG)+\nu_2(F_w F_0^m SG)\\
&=
\left(
\begin{array}{ccc}
1&1&1
\end{array}
\right)
E_0^m E_w
\left( \left(
\begin{array}{c}
6/5\\
2/5\\
2/5
\end{array}
\right)+
\left(
\begin{array}{c}
2/5\\
6/5\\
2/5
\end{array}
\right)+
\left(
\begin{array}{c}
2/5\\
2/5\\
6/5
\end{array}
\right)\right)\\
&=
\left(
\begin{array}{ccc}
1&1&1
\end{array}
\right)
E_0^m E_w
\left(
\begin{array}{c}
2\\2\\2
\end{array}
\right).\\
\end{split}
\end{equation*}

Now we let
$$
E_w
\left(
\begin{array}{c}
2\\
2\\
2
\end{array}
\right)=
\left(
\begin{array}{c}
x\\
y\\
z
\end{array}
\right)
$$
and see from the diagonalization of $E_0^m$ that
$$
\nu(F_w F_0^m SG)=
\left(
-\frac{3}{4}x+ \frac{9}{8}y+\frac{9}{8}z
\right)
\left(
\frac{1}{15}
\right)^m
+
\frac{1}{8}
\left(
14x-y-z
\right)
\left(
\frac{3}{5}
\right)^m
=
\Theta\left(
\frac{3}{5}
\right)^m
$$
since $14x-y-z>0$ by Lemma 3.2.
\end{proof}

We have proven the decay rates of the measures of a sequence of cells converging to a point which is not equivalent to a skew-symmetric cell. Surprisingly, in the case which the cell is equivalent to skew-symmetric at that point, we have the first condition for the Radon-Nikodym derivatives to be $0$.

\begin{lem}
\label{lemma3.4}
If $h$ is symmetric on the cell $F_w SG$ with respect to the point $F_w q_0$, then $\nu_h(F_2 F_1^m SG)=\Theta\left(\left(\frac{1}{15}\right)^m\right)$ and hence $\frac{d \nu_h}{d \nu}(F_w F_1 (q_2))=0$. Conversely, if $\frac{d \nu_h}{d \nu}(F_w F_1 (q_2))=0$, then $h$ is symmetric on the cell $F_w SG$ with respect to the point $F_w q_0$.
\end{lem}
\begin{proof}
We know that $\nu_h(F_w SG)=\left(\frac{5}{3}\right)^{|w|} \nu_{h \circ F_w}$, so we may consider the case when $w$ is the empty word. We see from Lemmas \ref{lemma3.1} and \ref{lemma3.3} that the term $\left(\frac{5}{3}\right)^m$ in $\nu_h (F_1F_2^m)$ vanishes if and only if $h$ is skew-symmetric on $F_1 SG$ with respect to the point $F_1 F_2^m (q_2)= F_1 (q_2)$,
 which is true if and only if $h$ is symmetric on $SG$ with respect to the point $q_0$,
 in which case $\nu_h(F_1 F_2^m SG)=\Theta\left(\left(\frac{1}{15}\right)^m\right)$ and the denominator of $\frac{d \nu_h}{d \nu}(x)=\lim \frac{\nu_h(F_2 F_1^m SG)}{\nu(F_2 F_1^m SG)}$ dominates. It follows that symmetry of $h$ around $q_0$ is sufficient and necessary to have $\frac{d \nu_h}{d \nu}(F_1 (q_2))=0$.
\end{proof}

\begin{thm}
Let $h$ be a harmonic function on SG with $\nu_h=a\nu_0+b\nu_1+c\nu_2$ and $h^\perp$ be a harmonic function orthonormal to $h$ under the energy inner product. Also let $C$ be a cell in $SG$. Then\\
a) $$\inf_{x \in C} \frac{d \nu_h}{d \nu}(x)=0;$$ \\
b)
$$
\sup_{x \in C} \frac{d \nu_h}{d \nu}=\frac{2}{3}(a+b+c).
$$
and if the maximum of $\frac{d\nu_h}{d \nu}$ is attained, then the minimum of $\frac{d\nu_{h^\perp}}{d \nu}$ is attained at the same point.
\end{thm}
\begin{proof}
We first prove a). If two of the values of $h$ on the boundary of $C$ are the same then $h$ is symmetric and the result follows from Lemma \ref{lemma3.4}, so we assume all three values are distinct. Since we consider $h$ modulo constants and the conclusion is unchanged by replacing $h$ with a constant multiple $ch$, we may assume without loss of generality that $h$ is harmonic with $h(q_0)=0$, $h(q_1)=1$, $h(q_2)=a\geq2$. Let $x$ be the point on the edge from $q_0$ to $q_1$ where $h$ achieves its maximum. It was shown in \cite{MaxOnEdge} that if we identify this edge in the obvious fashion with the unit interval $[0,1]$ oriented so $q_0$ corresponds to $0$ and $q_1$ to $1$, and define a map $M:[2,\infty)\to[0,1]$ by $M(a)=x$, then $M$ is continuous and strictly decreasing, and that $\lim_{a\to\infty}M(a)$ gives the midpoint $F_1q_0$. It follows that if we have a harmonic function on a cell such that the maximum along a side occurs at the midpoint of that side then the function is symmetric under the reflection fixing the vertex opposite the side. This is applicable in the case when the location of our maximum is $x\in V_*$ (equivalently a dyadic point of $[0,1]$ under our identification), because then there is a word $w$ such that $x=F_wF_1(q_0)$, and the function is harmonic on $F_w(SG)$ with its maximum along $F_w[q_0,q_1]$ at the midpoint of this side. The conclusion that $h$ is symmetric on $F_w(SG)$ allows us to apply Lemma \ref{lemma3.4} to deduce that $\frac{d\nu_h}{d\nu}(x)=0$.

Now suppose that $x$ is not a dyadic point. Since dyadic points are dense in the interval [0,1], we pick a increasing sequence $x_n$ of dyadic points converging to $x$. Because $M$ is decreasing and continuous, there exists a decreasing sequence of $a_n$ converging to $a$ satisfying $M(a_n) = x_n$.

Let $\varepsilon > 0$ be given. A simple computation shows that the energy measure $\nu_h = a\nu_0 + (1-a)\nu_1 + (a^2-a)\nu_2$. From the previous discussion, we can pick an $a_n$($x_n$, resp.), which will be denoted as $a_\varepsilon$($x_\varepsilon$, resp.), satisfying $|a-a_\varepsilon|+|(1-a)-(1-a_\varepsilon)|+|(a^2-a)-(a_\varepsilon^2-a_\varepsilon)|<\varepsilon$. We also denote the harmonic function with boundary values $0, 1, a_\varepsilon$ as $h_\varepsilon$. The previous discussion shows that $\frac{d\nu_{h_\varepsilon}}{d\nu}(x_\varepsilon) = 0$. Now,

\begin{equation*}
\begin{split}
&\frac{d\nu_h}{d\nu}(x_\varepsilon)\\
=& \frac{d(\nu_h-\nu_{h_\varepsilon})}{d\nu}(x_\varepsilon) + \frac{d\nu_{h_\varepsilon}}{d\nu}(x_\varepsilon)\\
=&\frac{d(\nu_h-\nu_{h_\varepsilon})}{d\nu}(x_\varepsilon)\\
=&\big{(}a-a_\varepsilon\big{)}\frac{d\nu_0}{d\nu}(x_\varepsilon)+\big{(}(1-a)-(1-a_\varepsilon)\big{)}\frac{d\nu_1}{d\nu}(x_\varepsilon)+\big{(}(a^2-a)-(a_\varepsilon^2-a_\varepsilon)\big{)}\frac{d\nu_2}{d\nu}(x_\varepsilon)\\
\leq & |a-a_\varepsilon|+|(1-a)-(1-a_\varepsilon)|+|(a^2-a)-(a_\varepsilon^2-a_\varepsilon)|\\
< & \varepsilon
\end{split}
\end{equation*}
where the penultimate inequality comes from the fact that $\frac{d\nu_i}{d\nu}\leq 1$ since $\nu_i(C)\leq \nu(C)$ for any cell $C$.
\\
And consequently,
$$\inf_{x\in SG} \frac{d\nu_h}{d\nu} = 0.$$
On a general cell $C$, $h|_C$ is harmonic, so the same argument applies.

We then show b). An easy computation shows $\frac{1}{3} \nu=\nu_{rh}+\nu_{h^\perp}$, where $r^2= \frac{1}{\nu_h(SG)}$. Since the total measure of $\nu_h$ on $SG$ is $2(a+b+c)$, we have
$$\frac{3}{2(a+b+c)}\frac{d\nu_h}{d\nu} + 3 \frac{d\nu_{h^\perp}}{d\nu} = 1.$$
It follows that
$$\frac{3}{2(a+b+c)}\sup_{x\in C} \frac{d\nu_h}{d\nu} = 1 - 3 \inf_{x \in C} \frac{d\nu_{h^\perp}}{d\nu} = 1.$$
Whence,
$$\sup_{x\in SG} \frac{d\nu_h}{d\nu} = \frac{2(a+b+c)}{3}.$$
Further, if the derivative attains its supremum, then the maximum occurs where $0$ occurs in $\frac{d\nu_{h^\perp}}{d\nu}$, that is, where the local extremum occurs in the smallest edge of the harmonic function $h^\perp$.
\end{proof}

As we can see, on every cell of SG, the supremum and infimum are distance $\frac{2(a+b+c)}{3}$ from each other, so the derivative is far from continuous. However, we will later recover limited continuity when we restrict to the edges of cells.

\section{Characterzation of Positive Energy Measures}
Within the 3-dimensional space of signed energy measures which measures are positive? The following result gives the answer.

\begin{thm}
Let $P$ be the set of $(a,b,c)\in \mathbb{R}^3$ such that  $a\nu_0+b\nu_1+c\nu_2$ is a positive energy measure. Then $P$ is the solid cone $S=\{(a,b,c)\in \mathbb{R}^3: ab+bc+ca\geq 0 \}$, and the boundary of $P$ is the set of $(a,b,c)$ such that $a\nu_0+b\nu_1+c\nu_2=\nu_h$ for some harmonic function $h$.
\end{thm}
\begin{proof}
We first show that $S \subset P$. We know from \cite{EMatrix} that the set of $(x,y,z)$ such that $x=\nu_h(F_0 SG),y=\nu_h(F_1 SG),z=\nu_h(F_2 SG)$ form the cone
$$
\frac{11}{25}(x+y+z)^2= x^2+y^2+z^2.
$$
A simple computation shows that
$$
\frac{2}{5}
\left  (
\begin{array}{ccc}
3&1&1\\
1&3&1\\
1&1&3
\end{array}
\right)
$$
transforms $(a,b,c)$ to the corresponding coefficients $(x,y,z)$. Hence, the set of $(a,b,c)$ that corresponds to some $\nu_h$ is the cone
$$
ab+bc+ca=0,
$$
which is the boundary of $S$. Then a general point in $S$ has the form $(a,b,c)+(\delta_0,\delta_1,\delta_2)$ with all $\delta_j\geq0$, so it yields a positive measure $\nu_h+\delta_0\nu_0+\delta_1\nu_1+\delta_2\nu_2$ and so belongs to $P$.

Now we show that the coefficients that come from harmonic functions (the boundary of $S$) are on the boundary of $P$. Consider the energy measure $\nu_h-\varepsilon \nu$ for $\varepsilon>0$. Suppose, for contradiction, that this measure is positive. Then for any cell $C$, $(\nu_h-\varepsilon\nu)(C)>0$, so $\frac{\nu_h(C)}{\nu(C)}>\varepsilon>0$ for all $C$, but we know the infimum should be zero, so this is a contradiction. Therefore $(a-\varepsilon, b-\varepsilon, c-\varepsilon)\notin P$ and $(a,b,c)$ is on the boundary of $P$. Since the two boundaries coincide, $S=P$.
\end{proof}

Immediately, since the solid cone is convex, every positive energy measure is precisely a convex combination of two energy measures, each coming from a single harmonic function. But we can prove a little bit more.
\begin{cor}
Every positive energy measure of harmonic functions is a convex combination of positive energy measures of $h$ and $h^\perp$ for some harmonic function $h$, and vice versa.
\end{cor}
\begin{proof}
For every harmonic function $h$, $\nu_h+\nu_{h^\perp}=c\nu$ for some $c>0$. This shows $\nu_h$, $\nu_{h^\perp}$, $\nu$ lie on the same two-dimensional subspace.

Suppose a two-dimensional subspace $W$ containing $\nu$ is given. $W$ intersects the boundary of the cone $S$ stated in the theorem on two and only two lines which contain energy measures of a single harmonic function. The above paragraph proves that one of the two lines contains $\nu_h$ for some $h$ while the other contains $\nu_{h^\perp}$.

Now, given any positive energy measure $\sigma$ which is not a multiple of $\nu$ (in which case, the conclusion is trivial), consider the two-dimensional subspace $V$ containing $\sigma$ and $\nu$. By the previous reasoning and the fact that $V\cap S$ is the closed convex hull of $V\cap \partial S$, there are harmonic functions $h$ and $h^{\perp}$ such that $\sigma$ is the convex combination of  $\nu_{ah}$ and $\nu_{bh^\perp}$ for some positive constants $a$ and $b$. It follows that for some $0\leq t\leq 1$, for any $r,s\geq0$,
$$\sigma = tr^{-2}\nu_{rah}+(1-t)s^{-2}\nu_{sbh^\perp}.$$
Choosing $r,s$ such that $tr^{-2}+(1-t)s^{-2}=1$ and $ra=sb$ concludes the proof.
\end{proof}


\section{Limited Continuity}
We have seen from the previous section that the derivative of an energy measure is not continuous; however, if we restrict the  derivative to the set of vertices $V_*$, it is continuous on the edges of every triangle. This can be seen in Figure \ref{cornerview}.

\begin{figure}
        \centering
        \includegraphics[height=4in]{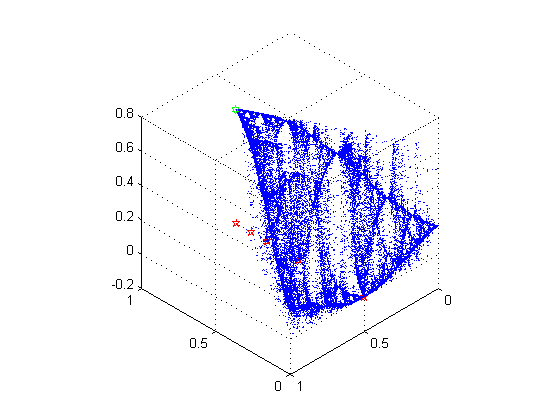}
        \caption{Radon-Nikodym Derivative of $\nu_0$ with red dots indicating minimum values and green dot indicating maximum value.}
        \label{cornerview}
 \end{figure}
 
By \eqref{eq:DiagE}, a simple computation shows that
\begin{equation}
\label{eq:LimitE}
\begin{split}
\lim_{m\to \infty} \left(\frac{5}{3}\right)^m E_0^m &= \frac{1}{40}
\left(\begin{array}{ccc}
42&-3&-3\\
14&-1&-1\\
14&-1&-1
\end{array}\right);\\
\lim_{m\to \infty} \left(\frac{5}{3}\right)^m E_1^m &= \frac{1}{40}
\left(\begin{array}{ccc}
-1&14&-1\\
-3&42&-3\\
-1&14&-1
\end{array}\right);\\
\lim_{m\to \infty} \left(\frac{5}{3}\right)^m E_2^m &= \frac{1}{40}
\left(\begin{array}{ccc}
-1&-1&14\\
-1&-1&14\\
-3&-3&42
\end{array}\right).\\
\end{split}
\end{equation}
Notice that all of the above are rank-one matrices.

We denote $A:=\lim_{m\to \infty} \left(\frac{5}{3}\right)^m E_2^m$ and $A_m:=\left( \frac{5}{3} \right)^m E_2^m, D_{m}:=A-A_m$.
We also know
$$E_wA = E_w\left(\begin{array}{c} 1\\1\\3
\end{array}\right)
\left(\begin{array}{ccc} -1&-1&14
\end{array}\right)$$
so
$$
\left(
\begin{array}{ccc}
-1&-1&14
\end{array}
\right)
E_w A
=
c_w\left(
\begin{array}{ccc}
-1 &-1 &14
\end{array}
 \right).
$$
In fact
$$
c_w=\left(
\begin{array}{ccc}
-1& -1& 14
\end{array}
 \right) E_w
 \left(
 \begin{array}{c}
 1\\
 1\\
 3
 \end{array}
 \right).
$$
We are going to prove, for any finite word $w$ consisting only $1$ and $2$, that $\left(\frac{5}{3}\right)^{|w|}c_w$\hspace{2pt} is bounded away from zero; to do this, we need the following lemmas.

\begin{lem}
\label{lemma5.1}
Given any finite word $w$ consists of letters $1$ and $2$ only, and let
$$\left(
\begin{array}{c}
x\\y\\z
\end{array}
\right)
=
\frac{5}{2}
E_{w}
\left(
\begin{array}{c}
\nu_2(F_0 SG)\\ \nu_2(F_1 SG)\\ \nu_2(F_2 SG)
\end{array}
\right)
=
E_{w}
\left(
\begin{array}{c}
1\\1\\3
\end{array}
\right).$$
Then
$$
\left(
\begin{array}{ccc}
-1&-1&14
\end{array}
\right)E_2
\left(
\begin{array}{c}
x\\y\\z
\end{array}
\right) \geq
\left(
\begin{array}{ccc}
-1&-1&14
\end{array}
\right)E_1
\left(
\begin{array}{c}
x\\y\\z
\end{array}
\right).$$
\end{lem}

\begin{proof}
We prove this by induction. It is trivial for $w$ being the empty word. Suppose that the conclusion holds for $|w|=m$, which is equivalent to saying $x-4y+11z\geq0$.
Also assume for induction that $z \geq x$.
Consider first the word $w1$. We have
\begin{equation*}
\begin{split}
\left(
\begin{array}{ccc}
-1&-1&14
\end{array}
\right)(E_2 - E_1)E_1
\left(
\begin{array}{c}
x\\y\\z
\end{array}
\right)&=
\frac{3}{5}\left(
\begin{array}{ccc}
1&-4&11
\end{array}
\right)E_1
\left(
\begin{array}{c}
x\\y\\z
\end{array}
\right)\\
&= \frac{1}{25}(-9x-4y+21z)\\
&= \frac{1}{25}\big{(}(x-4y+11z)+10(z-x)\big{)}\\
&\geq 0.
\end{split}
\end{equation*}
Furthermore, the third entry of $E_1 \left( \begin{array}{c}x\\y\\z\end{array} \right)$ minus the first entry is
$$15(z-x) \geq 0.$$
Now we turn to the second case concerning the word $w2$.
\begin{equation*}
\begin{split}
\left(
\begin{array}{ccc}
-1&-1&14
\end{array}
\right)(E_2 - E_1)E_2
\left(
\begin{array}{c}
x\\y\\z
\end{array}
\right)&=
\frac{3}{5}\left(
\begin{array}{ccc}
1&-4&11
\end{array}
\right)E_2
\left(
\begin{array}{c}
x\\y\\z
\end{array}
\right)\\
&= \frac{1}{5}(-3y+19z)\\
&= \frac{1}{5}\big{(}(x-4y+11z)+(z-x)+y+7z\big{)}\\
&\geq 0.
\end{split}
\end{equation*}\\
Also, the third entry of $E_2 \left( \begin{array}{c}x\\y\\z\end{array} \right)$ minus the first entry is
$$\frac{1}{25}\big{(}-4x+y+11z\big{)}=\frac{1}{25}\big{(}y+4(z-x)+7z\big{)}\geq 0.$$
The positivity of $y$ and $z$ comes from the fact that
$$
E_w \left(
\begin{array}{c}
1\\1\\3
\end{array}
\right) =
\frac{2}{5}
\left(
\begin{array}{c}
\nu_2(F_{w0} SG)\\ \nu_2(F_{w1} SG)\\ \nu_2(F_{w2} SG)
\end{array}
\right).
$$
This completes the induction.
\end{proof}

\begin{lem}
\label{lemma5.2}
Under the same hypothesis and notations as in lemma \ref{lemma5.1}, we have
$$
\left(
\begin{array}{ccc}
-1&-1&14
\end{array}
\right)
E_1 E_2 E_w
\left(
\begin{array}{c}
1\\1\\3
\end{array}
\right) \geq
\left(
\begin{array}{ccc}
-1&-1&14
\end{array}
\right)
E_2 E_1 E_w
\left(
\begin{array}{c}
1\\1\\3
\end{array}
\right).
$$
\end{lem}
\begin{proof}
Fix an $(m-2)$-length word $w$ consisting only letters $1$ and $2$. The column vector (1, 1, 3) corresponds to the energy measure $\nu_2$ obtained from only $h_2$.
By modulo a constant, we have, as shown in the figures of $F_w SG$, the values of $h_2$ and energies on the cells of $\nu_2$. Note that $0\leq a\leq b$.

\begin{center}
\includegraphics[scale=.5]{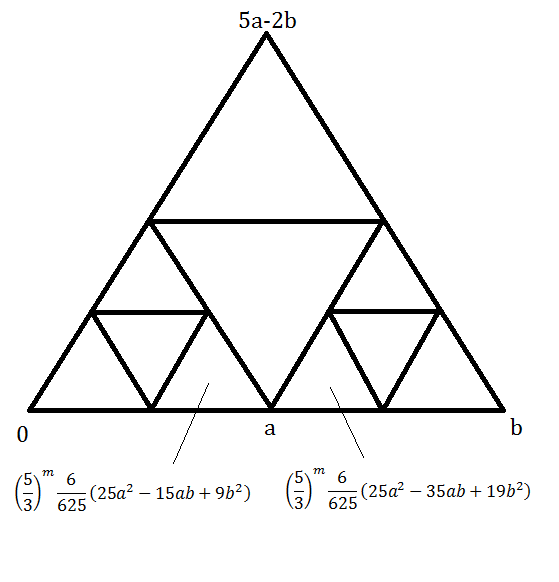}\\
\includegraphics[scale=.5]{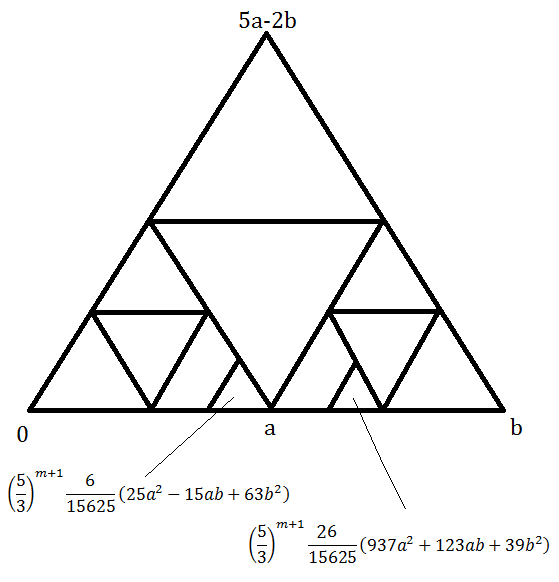}
\end{center}

A simple manipulation on matrices shows that
\begin{equation*}
\begin{split}
&\hspace{2pt} \left(
\begin{array}{ccc}
-1&-1&14
\end{array}
\right)
E_2 E_1 E_w
\left(
\begin{array}{c}
1\\1\\3
\end{array}
\right)\\
= & \hspace{2pt}
15
\left(
\begin{array}{ccc}
1&1&1
\end{array}
\right)
E_2 E_2 E_1 E_w
\left(
\begin{array}{c}
1\\1\\3
\end{array}
\right)
-
\left(
\begin{array}{ccc}
1&1&1
\end{array}
\right)
E_2 E_1 E_w
\left(
\begin{array}{c}
1\\1\\3
\end{array}
\right)\\
=& \hspace{2pt}\frac{5}{2} \left(15\nu_2(F_w F_1 F_2 F_2 SG) - \nu_2(F_w F_1 F_2 SG)\right).
\end{split}
\end{equation*}

Similarly,
\begin{equation*}
\begin{split}
&\hspace{2pt} \left(
\begin{array}{ccc}
-1&-1&14
\end{array}
\right)
E_1 E_2 E_w
\left(
\begin{array}{c}
1\\1\\3
\end{array}
\right)\\
= & \hspace{2pt}
15
\left(
\begin{array}{ccc}
1&1&1
\end{array}
\right)
E_2 E_1 E_2 E_w
\left(
\begin{array}{c}
1\\1\\3
\end{array}
\right)
-
\left(
\begin{array}{ccc}
1&1&1
\end{array}
\right)
E_1 E_2 E_w
\left(
\begin{array}{c}
1\\1\\3
\end{array}
\right)\\
=& \hspace{2pt}\frac{5}{2} \left(15\nu_2( F_w F_2 F_1 F_2 SG) - \nu_2( F_w F_2 F_1 SG)\right).
\end{split}
\end{equation*}
Whence, from the calculation in the figures,
\begin{equation*}
\begin{split}
&\left(
\begin{array}{ccc}
-1&-1&14
\end{array}
\right)
E_1 E_2 E_w
\left(
\begin{array}{c}
1\\1\\3
\end{array}
\right)
-
\left(
\begin{array}{ccc}
-1&-1&14
\end{array}
\right)
E_2 E_1 E_w
\left(
\begin{array}{c}
1\\1\\3
\end{array}
\right)\\
=
&\frac{4}{625} \left(\frac{5}{3}\right)^m(6053a^2 + 852ab + 144b^2)\geq 0.
\end{split}
\end{equation*}
\end{proof}

What we have proven in the above lemmas is that, for two level-$m$ cells $F_w SG$ and $F_{w'} SG$ whose words consist only of the letters $1$ and $2$, we have $\nu_2(F_w SG) \leq \nu_2(F_{w'} SG)$, provided that $F_w SG$ is to the left of $F_{w'}$, in the orientation in which $q_1$ is the left endpoint of $[q_1,q_2]$. 
It also follows immediately that for those word $w$ of length $m$, $$c_w \geq c_{1,...,1} = \frac{15}{5^m}+\frac{99}{4 \cdot 15^m}+\frac{5}{2} \cdot \left( \frac{3}{5}\right)^m$$
which implies $\left(\frac{5}{3}\right)^{|w|} c_w$ is bounded away from zero, for every such $w$.\\
We now turn to an upper bound for $\left( \frac{5}{3} \right)^{|w|} \|E_w\|$, where $\| \cdot \|$ is the operator $1-$norm.

\vspace{15pt}
\begin{lem}
For any harmonic function $f$ and word $w$ with $|w|=m$, $$\mathcal{E}_{F_w}(f) \leq 2 \hspace{1pt} \textrm{Osc}(f,SG)^2 \left( \frac{3}{5}\right)^m.$$
\end{lem}
\begin{proof} By \cite{Bob},
$$\textrm{Osc}(f, F_w SG) \leq \left( \frac{3}{5} \right)^m \textrm{Osc}(f,SG).$$
Since for fixed oscillation, the greatest energy is attained by a function which is symmetric on $F_w SG$, which has energy $\left(\frac{5}{3} \right)^m 2\left(\mathrm{Osc}(f, F_w SG \right)^2$, we have
$$\mathcal{E}_{F_w}(f) \leq 2 \hspace{1pt} \textrm{Osc}(f,SG)^2 \left( \frac{3}{5}\right)^m.$$
\end{proof}

\begin{lem}
\label{lemma5.4}
There exists a constant $C$, independent of $m$, such that $$\| E_w \| \leq C \left( \frac{3}{5}\right)^m$$
for all $|w| = m$
\end{lem}
\begin{proof} Consider $\|(x,y,z)\| = 1$. If we view $x,y,z$ as the energies of the level 1 cells, we know that $(x,y,z)$ corresponds to an energy measure $a_0 \nu_0 + a_1 \nu_1 + a_2 \nu_2$.
Since $(x,y,z)$ and $(a_0, a_1, a_2)$ differ only by a linear transformation, $\|(a_0, a_1, a_2)\| \leq B$ for some constant $B$ for all such triples $(a_0,a_1,a_2)$.
Let a word $w$ be fixed, with $|w| = m$. Since from Lemma 5.3 we know that $\nu_j(F_{wi}SG)\leq 2 \left( \frac{3}{5} \right)^{m+1}$, we have
\begin{equation*}
\begin{split}
\bigg{\|} E_w \left( \begin{array}{c} x\\y\\z \end{array} \right) \bigg{\|} &= \sum_{i=0}^2 \bigg{|}\sum_{j=0}^2 a_j \nu_j (F_{wi} SG) \bigg{|}\\
&\leq \sum_{i=0}^2 \sum_{j=0}^2 \big{|} a_j \big{|} \big{|} \nu_j (F_{wi} SG) \big{|}\\
&\leq \sum_{i=0}^2 \sum_{j=0}^2 \big{|} a_j \big{|} \dot{} 2 \left( \frac{3}{5} \right)^{m+1}\\
&\leq 6B \left( \frac{3}{5} \right)^{m+1}.\\
\end{split}
\end{equation*}
Taking supremum over all $\|(x,y,z)\| = 1$ and letting $C = 6B \left( \frac{3}{5} \right)$ give
$$
\|E_w\| \leq C \left( \frac{3}{5} \right)^{m}.
$$
\end{proof}

We now give the main result of this section, the continuity on the three edges of a cell.
\begin{thm}
\label{thm:LimitedContinuity}
Given an energy measure $\nu_{f,g}$ and a cell $F_w SG$, the restriction of the derivative $\frac{d\nu_{f,g}}{d\nu}$ on the three edges restricted to the vertices $V_*$ is continuous.
\end{thm}
\begin{proof}
Since every point $x$ of $V_*$ that is on the edge of a cell is the junction point of two cells (boundary points excepted)
it suffices to show that the restriction of the derivative $\frac{d\nu_{f,g}}{d\nu}$ on the edge $F_w [q_1, q_2]$ is continuous at $F_w q_2$, where $[q_1, q_2]$ is the edge connecting $q_1$ and $q_2$. Continuity on the other side of $F_w q_2$ follows from symmetry and arbitrary choice of cell. Let

$$X=\left(
\begin{array}{c}
\nu_{f,g}(F_{w0}SG)\\ \nu_{f,g}(F_{w1}SG)\\ \nu_{f,g}(F_{w2}SG)\\
\end{array}\right),\qquad
\Xi=\left(\begin{array}{c} \alpha\\ \beta\\ \gamma\end{array}\right)=\left(
\begin{array}{c}
\nu(F_{w0}SG)\\ \nu(F_{w1}SG)\\ \nu(F_{w2}SG)\\
\end{array}\right).
$$
The points of $V_*$ in the interval $F_wF_2^n[q_1,q_2]$ with the exception of $F_wF_2^n(q_1)$ are of the form $F_wF_2^nF_{w'}q_2$. We know that the Radon-Nikodym derivative at these points converges to $\frac{d\nu_{f,g}}{d\nu}(F_wq_2)$ uniformly in $w'$, establishing continuity of the derivative at $F_wq_2$. To do so, fix $w'$ with $|w'|=m$ and observe that
$$\frac{d\nu_{f,g}}{d\nu}(F_wF_2^nF_{w'}q_2)=\frac{
\left(\begin{array}{ccc}-1&-1&14\end{array}\right)E_{w'}A_nX
}{
\left(\begin{array}{ccc}-1&-1&14\end{array}\right)E_{w'}A_n\Xi
}=\frac{
\left(\begin{array}{ccc}-1&-1&14\end{array}\right)E_{w'}(A-D_n)X
}{
\left(\begin{array}{ccc}-1&-1&14\end{array}\right)E_{w'}(A-D_n)\Xi
}.$$
Since $A$ is rank $1$ we have
$$\frac{
\left(\begin{array}{ccc}-1&-1&14\end{array}\right)E_{w'}AX
}{
\left(\begin{array}{ccc}-1&-1&14\end{array}\right)E_{w'}A\Xi
}=\frac{
\left(\begin{array}{ccc}-1&-1&14\end{array}\right)X
}{
\left(\begin{array}{ccc}-1&-1&14\end{array}\right)\Xi
}=\frac{d\nu_{f,g}}{d\nu}(F_wq_2).$$
Also, from Lemmas \ref{lemma5.1} and \ref{lemma5.2} we know $\left(\frac{5}{3}\right)^m\left(\begin{array}{ccc}-1&-1&14\end{array}\right)E_{w'}A\Xi \geq \frac{5}{2}(14\gamma-\alpha-\beta)>0$, while for sufficiently large $n$ independent of $m$ we have $\|E_{w'}D_n\|\leq \delta\left(\frac{3}{5}\right)^m$ because of Lemma \ref{lemma5.4} and the fact that $A_n\to A$ in operator norm. The result follows.

\end{proof}

The above idea uses only the lower bound for
$\left(\frac{5}{3}\right)^m \left( \begin{array}{ccc} -1&-1&14 \end{array} \right) E_w \left( \begin{array}{c} 1\\1\\3 \end{array} \right)$ and upper bound for $\left(\frac{5}{3}\right)^m \|E_w\|$.
\section{Average Value of the Radon-Nikodym Derivative}
We have shown how to calculate the values of the derivative on the vertices in $V_*$; moreover, we have continuity on the edges of every triangle. We would like to relate the average of derivatives on the whole cell to a weighted average on three vertices of the cell; that is, to find, on a cell $C$, positive real numbers $\{b_j\}_{j=0}^2$ whose sum is $1$ satisfying
$$\frac{\nu_{f,g}(C)}{\nu(C)}=\frac{1}{\nu(C)}\int_C{\frac{d\nu_{f,g}}{d\nu}d\nu}=\sum_{r_j \in \partial C}{b_j \frac{d\nu_{f,g}}{d\nu}(r_j)}$$
for every energy measure of harmonic functions $\nu_{f,g}$. In this section, we will show that the coefficients $b_j$ exist and discover some characterizations of them.
\subsection{Existence}
Recall from Theorem \ref{thm:MMatrices}, that we have
$$
\left(
\begin{array}{c}
\nu_0\\ \nu_1\\ \nu_2
\end{array}
\right)=
M_0 \left(
\begin{array}{c}
\nu_0\\ \nu_1\\ \nu_2
\end{array}
\right) \circ F_0^{-1}
+
M_1 \left(
\begin{array}{c}
\nu_0\\ \nu_1\\ \nu_2
\end{array}
\right) \circ F_1^{-1}
+
M_2 \left(
\begin{array}{c}
\nu_0\\ \nu_1\\ \nu_2
\end{array}
\right) \circ F_2^{-1}
$$
and
$$
\left(
\begin{array}{c}
\nu_0 (F_j A)\\ \nu_1 (F_j A)\\ \nu_2 (F_j A)
\end{array}
\right) =
M_j \left(
\begin{array}{c}
\nu_0 ( A)\\ \nu_1 ( A)\\ \nu_2 (A)
\end{array}
\right).
$$
Letting $M_w=M_{w_1}M_{w_2}\cdots M_{w_m}$ we see that
$$
\left(
\begin{array}{c}
\nu_0 (F_w A)\\ \nu_1 (F_w A)\\ \nu_2 (F_w A)
\end{array}
\right) =
M_w \left(
\begin{array}{c}
\nu_0 ( A)\\ \nu_1 ( A)\\ \nu_2 (A)
\end{array}
\right).
$$

\begin{thm}
On each cell $C=F_w SG$, there exist unique $(b_0, b_1, b_2)\in \mathbb{R}^3$ with $b_0+b_1+b_2=1$ such that
$$\textrm{Avg}_{F_w SG}\frac{d\nu_{f,g}}{d\nu}=\frac{1}{\nu(C)}\int_C\frac{d\nu_{f,g}}{d\nu}d\nu=\sum_{j=0}^2{b_j\frac{d\nu_{f,g}}{d\nu}\big{(}F_w(q_j)\big{)}}$$
for all energy measure (not necessarily positive) $\nu_{f,g}$ on $SG$. Explicitly,
\begin{equation}
\label{eq:6.1}
b_j= \frac{1}{6}+ \frac{1}{2}\frac{\text{sum of column }j \text{ of }M_w}{\text{sum of all entries of } M_w}.
\end{equation}
\end{thm}

\begin{proof}
We first introduce some notation. We let

\begin{equation*}
\begin{split}
z_0=M_w \left( \begin{array}{c}
1\\ 0\\ 0
\end{array}
\right)
,
z_1&=M_w \left( \begin{array}{c}
0\\ 1\\ 0
\end{array}
\right)
,
z_2=M_w \left( \begin{array}{c}
0\\ 0\\ 1
\end{array}
\right)
,
z=z_0+z_1+z_2.
\end{split}
\end{equation*}

Now suppose $\nu_{f,g}=\alpha_0 \nu_0+\alpha_1 \nu_1+\alpha_2 \nu_2$. Then we have
$$
\nu_{f,g} (F_w SG)= 2\left(
\begin{array}{ccc}
\alpha_0&\alpha_1&\alpha_2
\end{array}
\right)
z,\hspace{10pt}
\nu (F_w SG)= 2\left(
\begin{array}{ccc}
1&1&1
\end{array}
\right)
z
$$
so the average value of the Radon-Nikodym derivative of $\nu_{f,g}$ on the cell $F_w SG$ is
$$
\frac{\int_{F_w SG} \frac{d \nu_{f,g}}{d\nu} d\nu}{\int_{F_w SG} d \nu}= \frac{\nu_{f,g} (F_w SG)}{\nu(F_w SG)}=\frac{\left(
\begin{array}{ccc}
\alpha_0&\alpha_1&\alpha_2
\end{array}
\right)
z}{\left(
\begin{array}{ccc}
1&1&1
\end{array}
\right)
z}.
$$

We also note that
\begin{equation*}
\begin{split}
\frac{d\nu_{f,g}}{d\nu}(F_w q_j)&= \lim_{m \to \infty} \frac{\nu_{f,g} (F_w F_j^m SG)}{\nu (F_w F_j^m SG)}\\
&= \lim_{m \to \infty}
\frac{\left(
\begin{array}{ccc}
\alpha_0&\alpha_1&\alpha_2
\end{array}
\right)
M_w
\left(
\begin{array}{c}
\frac{d\nu_0}{d\nu}(q_j)\\\frac{d\nu_1}{d\nu}(q_j)\\\frac{d\nu_2}{d\nu}(q_j)
\end{array}
\right)}{\left(
\begin{array}{ccc}
1&1&1
\end{array}
\right)
M_w
\left(
\begin{array}{c}
\frac{d\nu_0}{d\nu}(q_j)\\\frac{d\nu_1}{d\nu}(q_j)\\\frac{d\nu_2}{d\nu}(q_j)
\end{array}
\right)}
\end{split}
\end{equation*}
and we know that
$\frac{d\nu_{j}}{d\nu}(q_k)$ is equal to $\frac{2}{3}$ if $j=k$ and $\frac{1}{6}$ otherwise.

So
\begin{equation*}
\begin{split}
\frac{d\nu_{f,g}}{d\nu}(F_w q_0)&=
\frac{\left(
\begin{array}{ccc}
\alpha_0&\alpha_1&\alpha_2
\end{array}
\right)
M_w
\left(
\begin{array}{c}
\frac{2}{3}\\\frac{1}{6}\\\frac{1}{6}
\end{array}
\right)}{\left(
\begin{array}{ccc}
1&1&1
\end{array}
\right)
M_w
\left(
\begin{array}{c}
\frac{2}{3}\\\frac{1}{6}\\\frac{1}{6}
\end{array}
\right)} \\
&= \frac{
\left(
\begin{array}{ccc}
\alpha_0&\alpha_1&\alpha_2
\end{array}
\right)
\left(
z_0 + \frac{1}{3}z
\right)
}{
\left(
\begin{array}{ccc}
1&1&1
\end{array}
\right)
\left(
z_0 + \frac{1}{3}z
\right)
};
\end{split}
\end{equation*}
and similarly
\begin{equation*}
\begin{split}
\frac{d\nu_{f,g}}{d\nu}(F_w q_1)
&= \frac{
\left(
\begin{array}{ccc}
\alpha_0&\alpha_1&\alpha_2
\end{array}
\right)
\left(
z_1 +\frac{1}{3} z
\right)
}{
\left(
\begin{array}{ccc}
1&1&1
\end{array}
\right)
\left(
z_1 + \frac{1}{3}z
\right)
};\\
\frac{d\nu_{f,g}}{d\nu}(F_w q_2)
&= \frac{
\left(
\begin{array}{ccc}
\alpha_0&\alpha_1&\alpha_2
\end{array}
\right)
\left(
z_2 + \frac{1}{3} z
\right)
}{
\left(
\begin{array}{ccc}
1&1&1
\end{array}
\right)
\left(
z_2 + \frac{1}{3}z
\right)
}
.
\end{split}
\end{equation*}

Now we suppose that we can write the average of the derivative on the whole cell as a weighted average of the value of the derivative on the boundary points; that is, we write

\begin{equation*}
\begin{split}
\textrm{Avg}_{F_w SG}\frac{d\nu_{f,g}}{d\nu}
&=\frac{\left(
\begin{array}{ccc}
\alpha_0&\alpha_1&\alpha_2
\end{array}
\right)z
}{\left(
\begin{array}{ccc}
1&1&1
\end{array}
\right)z} \\
&= b_0 \frac{d\nu_{f,g}}{d\nu}(F_w q_0) + b_1 \frac{d\nu_{f,g}}{d\nu}(F_w q_1) + b_2 \frac{d\nu_{f,g}}{d\nu}(F_w q_2)\\
&= \left(
\begin{array}{ccc}
\alpha_0&\alpha_1&\alpha_2
\end{array}
\right)
\left(
\sum_{j=0}^2{ \frac{b_j\left(
z_j +\frac{1}{3} z
\right)
}{
\left(
\begin{array}{ccc}
1&1&1
\end{array}
\right)
\left(
z_j + \frac{1}{3}z
\right)}}
\right).
\end{split}
\end{equation*}
Then it is implied that
\begin{equation}
\label{eq:6.2}
\sum_{j=0}^2{ \frac{b_j\left(
z_j +\frac{1}{3} z
\right)
}{
\left(
\begin{array}{ccc}
1&1&1
\end{array}
\right)
\left(
z_j + \frac{1}{3}z
\right)}}
=
\frac{z}{\left(\begin{array}{ccc}1&1&1 \end{array}\right)z}.
\end{equation}

We also see that multiplying both sides of this equation by $\left( \begin{array}{ccc}
1&1&1
\end{array}\right)$ gives $b_0+b_1+b_2=1$.

The solution to \eqref{eq:6.2} is
\begin{equation*}
\begin{split}
b_j&= \frac{\left(\begin{array}{ccc}
1&1&1
\end{array} \right)
\left( \frac{1}{6}z + \frac{1}{2}z_j \right)
}{
\left( \begin{array}{ccc}
1&1&1
\end{array}
\right)z},
\end{split}
\end{equation*}
as can be seen by substitution. The solution is unique because the construction of the $z_j+\frac{1}{3}z$ from the invertible matrix $M_w$ ensures that they are linearly independent.

Noting that $\left( \begin{array}{ccc}
1&1&1
\end{array} \right)z_j$ is the sum of the entries in column $j$ of $M_w$ and $\left( \begin{array}{ccc}
1&1&1
\end{array} \right)z$ is the sum of all entries of $M_w$. This yields \eqref{eq:6.1}.
\end{proof}
Later we will show that the $b_j$ values lie between $0$ and $\frac{2}{3}$; thus, the above theorem shows that we can relate the average value of the derivative on the whole cell to the weighted average value of the derivative at the vertices of the cell(one set of weights works for all energy measures). Since the $b_j$ values depend on the word $w$, from now on we will denote $b_j$ on $F_w SG$ as $b_j^{(w)}$.

\subsection{Characterization}
We now have coefficients $b_j^{(w)}$ for the cell $F_w SG$ to calculate the average of derivatives on the whole cell with the values on vertices of the cell.  Next, we show they can be calculated recursively from cells to subcells, or alternatively using the Kusuoka measure.

The following theorem gives the first way to calculate all the $b$'s by calculating them recursively.
\begin{thm}
\label{th:bRelation}
Using the previous notations, we have the following relationship from cell to subcells:
\begin{equation*}
\begin{array}{ccc}
b_0^{(w0)}=\frac{9b_0^{(w)}}{13b_0^{(w)}+b_1^{(w)}+b_2^{(w)}}
&b_1^{(w0)}=\frac{2b_0^{(w)}+2b_1^{(w)}-b_2^{(w)}}{13b_0^{(w)}+b_1^{(w)}+b_2^{(w)}}
&b_2^{(w0)}=\frac{2b_0^{(w)}-b_1^{(w)}+2b_2^{(w)}}{13b_0^{(w)}+b_1^{(w)}+b_2^{(w)}}
\\
b_0^{(w1)}=\frac{2b_0^{(w)}+2b_1^{(w)}-b_2^{(w)}}{b_0^{(w)}+13b_1^{(w)}+b_2^{(w)}}
&b_1^{(w1)}=\frac{9b_1^{(w)}}{b_0^{(w)}+13b_1^{(w)}+b_2^{(w)}}
&b_2^{(w1)}=\frac{-b_0^{(w)}+2b_1^{(w)}+2b_2^{(w)}}{b_0^{(w)}+13b_1^{(w)}+b_2^{(w)}}
\\
b_0^{(w2)}=\frac{2b_0^{(w)}-b_1^{(w)}+2b_2^{(w)}}{b_0^{(w)}+b_1^{(w)}+13b_2^{(w)}}
&b_1^{(w2)}=\frac{-b_0^{(w)}+2b_2^{(w)}+2b_1^{(w)}}{b_0^{(w)}+b_1^{(w)}+13b_2^{(w)}}
&b_2^{(w2)}=\frac{9b_2^{(w)}}{b_0^{(w)}+b_1^{(w)}+13b_2^{(w)}}
\end{array}
\end{equation*}
\end{thm}

\begin{proof}
We know that, on $F_w SG$,
$$b_j^{(w)} = \frac{1}{6} + \frac{1}{2}\frac{c_j^{(w)}}{c_0^{(w)}+c_1^{(w)}+c_2^{(w)}}$$
where $\left( \begin{array}{ccc} c_0^{(w)}&c_1^{(w)}&c_2^{(w)} \end{array}\right) = \left( \begin{array}{ccc} 1&1&1 \end{array}\right) M_w$.
Put $a_j = 2(b_j^{(w)} - \frac{1}{6})$. We see that $a_j(c_0^{(w)}+c_1^{(w)}+c_2^{(w)})=c_j^{(w)}$. Since
$$\left( \begin{array}{ccc} c_0^{(w)}&c_1^{(w)}&c_2^{(w)} \end{array}\right)M_0 = \frac{1}{15}\left( \begin{array}{ccc} 9c_0^{(w)}+2c_1^{(w)}+2c_2^{(w)}& 2c_1^{(w)}-c_2^{(w)}& 2c_2^{(w)}-c_1^{(w)} \end{array}\right),$$
we have
\begin{align*}
b_0^{(w0)} &= \frac{1}{6}+\frac{1}{2}\frac{9c_0^{(w)}+2c_1^{(w)}+2c_2^{(w)}}{9c_0^{(w)}+3c_1^{(w)}+3c_2^{(w)}}\\ &= \frac{1}{6}+\frac{1}{2}\frac{9a_0+2a_1+2a_2}{9a_0+3a_1+3a_2}.
\end{align*}
Similarly, we have
$$b_1^{(w0)} = \frac{1}{6}+\frac{1}{2}\frac{2a_1-a_2}{9a_0+3a_1+3a_2}$$
and
$$b_2^{(w0)} = \frac{1}{6}+\frac{1}{2}\frac{2a_2-a_1}{9a_0+3a_1+3a_2}.$$
With some computation from the definition of the $a$ values and using the fact that $\sum_j b_j^{(w)}=1$, we can relate $b_j^{(w)}$ and $b_j^{(w0)}$ explicitly as in the statement of the theorem.

Observe that multiplying different $M_j$'s only differs by some permutation of subscript and entries, we have the rest of the results.
\end{proof}
The next theorem presents the second way to calculate $b_j^{(w)}$.
\begin{thm}
\label{th:SameRatio}
The distance of $b_j^{(w)}$ from $1/3$ is proportional to how ``skewed" the Kusuoka measure is on the cell $F_w F_j SG$ relative to $F_w SG$. Specifically,
$$\frac{1}{5} \left( b_j^{(w)} - \frac{1}{3} \right) = \frac{1}{4} \left( \frac{\nu(F_w F_j SG)}{\nu (F_w SG)} - \frac{1}{3} \right).$$
\end{thm}
\begin{proof}
By \eqref{eq:6.1} we have
$$b_j^{(w)}=\frac{1}{6}+\frac{
\left(\begin{array}{ccc}1&1&1\end{array}\right)z_j}{
\left(\begin{array}{ccc}1&1&1\end{array}\right)z
}.$$
Note that $\left(\begin{array}{ccc}1&1&1\end{array}\right)z=\nu(F_wSG)$. For simplicity take $j=0$. Then

%

\begin{equation*}
\begin{split}
\left(\begin{array}{ccc}1&1&1\end{array}\right)z_0
= &\frac{5}{4}
\left(
\begin{array}{ccc}
1&1&1
\end{array}
\right)
M_w
\left(
\begin{array}{c}
4/5\\
0\\
0
\end{array}
\right)\\
=&
\frac{5}{4}
\left(
\left(
\begin{array}{ccc}
1&1&1
\end{array}
\right)
M_w
\left(
\begin{array}{c}
6/5\\
2/5\\
2/5
\end{array}
\right)
-
\left(
\begin{array}{ccc}
1&1&1
\end{array}
\right)
M_w
\left(
\begin{array}{c}
2/5\\
2/5\\
2/5
\end{array}
\right)
\right) \\
=& \frac{5}{4} \left( \nu (F_w F_0 SG) - \frac{1}{5} \nu(F_w SG) \right)
\end{split}
\end{equation*}
where the last equality comes from noting that $6/5=\nu_0(F_0 SG)$ and $2/5=\nu_1(F_0 SG)=\nu_2(F_0 SG)$. More generally we have

$$b_j^{(w)} = \frac{5}{4} \left( \frac{2}{15} + \frac{\nu(F_w F_j SG) - \frac{1}{5}\nu (F_w SG)}{\nu (F_w SG)}\right) = \frac{5}{4} \left( \frac{-1}{15} + \frac{\nu(F_w F_j SG)}{\nu (F_w SG)}\right)$$
which gives
$$\frac{1}{5} \left( b_j^{(w)} - \frac{1}{3} \right) = \frac{1}{4} \left( \frac{\nu(F_w F_j SG)}{\nu (F_w SG)} - \frac{1}{3} \right).$$
\end{proof}
As a consequence of the theorem, the bounds for $(b_j^{(w)}-\frac{1}{3})$ also give bounds for how skewed the Kusuoka measure is in a cell. We now provide some bounds for the coefficients $b_j^{(w)}$.

\begin{thm}
a) The infimum of $\{b_j^{(w)}\}$ over all words $w$ is 0.\\
b) The supremum of $\{b_j^{(w)}\}$ over all words $w$ is $\frac{2}{3}$.
\end{thm}
\begin{proof}
We first prove a). We know that
\begin{equation*}
\frac{1}{15}\left(
\begin{array}{ccc}
14 &-1&-1
\end{array}
\right)E_w\left(
\begin{array}{c}
2\\
2\\
2
\end{array}
\right)
=
\frac{1}{15}
\lim_m \left(\frac{5}{3}\right)^m \nu(F_w F_0^m)\\
\geq 0
\end{equation*}
and the inequality is sharp over all $w$ because every energy measure is non-atomic.
So
\begin{equation*}
\begin{split}
&\frac{1}{15}
\left(
\begin{array}{ccc}
14 &-1&-1
\end{array}
\right)E_w\left(
\begin{array}{c}
2\\
2\\
2
\end{array}
\right)\\
=&
\left(
\begin{array}{ccc}
1 &1&1
\end{array}
\right)(E_0-\frac{1}{15}I)E_w\left(
\begin{array}{c}
2\\
2\\
2
\end{array}
\right)\\
=&
\left(
\begin{array}{ccc}
1&1&1
\end{array}
\right)E_0 E_w \left( \begin{array}{c}
2\\
2\\
2
\end{array} \right)
-
\frac{1}{15}
\left(
\begin{array}{ccc}
1&1&1
\end{array}
\right)E_w \left( \begin{array}{c}
2\\
2\\
2
\end{array} \right)\\
\geq& 0
\end{split}
\end{equation*}
and hence
\begin{equation*}
\begin{split}
&\left(
\begin{array}{ccc}
1&1&1
\end{array}
\right)E_0 E_w \left( \begin{array}{c}
2\\
2\\
2
\end{array} \right)
-
\frac{1}{5}
\left(
\begin{array}{ccc}
1&1&1
\end{array}
\right)E_w \left( \begin{array}{c}
2\\
2\\
2
\end{array} \right)\\
\geq&
-
\frac{2}{15}
\left(
\begin{array}{ccc}
1&1&1
\end{array}
\right)E_w \left( \begin{array}{c}
2\\
2\\
2
\end{array} \right)
\end{split}
\end{equation*}
which implies
\begin{equation*}
\frac{
\nu(F_w F_0 SG) - \frac{1}{5}\nu (F_w SG)
}{
\nu(F_w SG)
}
\geq -\frac{2}{15}
\end{equation*}
and so finally we get that
\begin{equation*}
b_0^{(w)}=\frac{5}{4}\left(\frac{2}{15} + \frac{\nu(F_w F_0 SG)-\frac{1}{5}\nu(F_w SG)}{\nu(F_w SG)} \right)\geq 0.
\end{equation*}
Since the inequality is sharp,
 zero is the infimum, completing a).\\
To prove b), we see that
\begin{equation*}
\begin{split}
\left(
\begin{array}{ccc}
-2&3&3
\end{array}
\right)E_w \left(
\begin{array}{c}
2\\
2\\
2
\end{array}
\right)&=
\frac{1}{5}
\left(
\begin{array}{ccc}
-1&-1&14
\end{array}
\right)E_1 E_w \left(
\begin{array}{c}
2\\
2\\
2
\end{array}
\right)\\
&=
\frac{1}{5}\lim_m \left( \frac{5}{3} \right)^m \nu (F_{w,1} F_2^m) \geq 0
\end{split}
\end{equation*}
and this inequality is sharp.
So
\begin{equation*}
\begin{split}
5\nu(F_w F_0 SG)-3\nu(F_w SG) &=
\left(
\begin{array}{ccc}
1& 1& 1
\end{array}
\right) (5E_0 - 3I)E_w \left(
\begin{array}{c}
2\\
2\\
2
\end{array}
\right)\\ &= \left(
\begin{array}{ccc}
2 -3 -3
\end{array}
\right) E_w \left(
\begin{array}{c}
2\\
2\\
2
\end{array}
\right)\leq 0.
\end{split}
\end{equation*}
Whence
$$
\frac{5}{2}(\nu(F_w F_0 SG)-\frac{1}{5}\nu(F_w SG)) \leq \nu(F_w SG)
$$
which implies
$$
\frac{5}{2}\frac{\nu(F_w F_0 SG)-\frac{1}{5}\nu(F_w SG)}{\nu(F_w SG)}\leq 1
$$
and so finally we have
$$
b_0^{(w)}=\frac{1}{6}+\frac{5}{4}\frac{\nu(F_w F_0 SG)-\frac{1}{5}\nu(F_w SG)}{\nu(F_w SG)}\leq \frac{2}{3}.
$$
Since the inequality is sharp,
this is the supremum. We have b).
\end{proof}
The previous results concern the bounds for a single value $b_j^{(w)}$.It is natural to consider $(b_0^{(w)}, b_1^{(w)}, b_2^{(w)})$ as a vector in $\mathbb{R}^3$, lying in the plane $\{(x,y,z):x+y+z=1\}$. Surprisingly, these vectors all lie in the disk centered at $(\frac{1}{3},\frac{1}{3},\frac{1}{3})$ of radius $\frac{1}{\sqrt{6}}$.
\begin{thm}
\label{thm:bValueCircle}
For all words $w$,
$$\sum\left(b_j^{(w)}-\frac{1}{3} \right)^2 < \frac{1}{6}$$ and this inequality is sharp.
\end{thm}
\begin{proof}
Let $\left(
\begin{array}{ccc}
1& 1& 1
\end{array}
 \right) M_w =
 \left(
\begin{array}{ccc}
c_0^{(w)}& c_1^{(w)}& c_2^{(w)}
\end{array}
 \right).$ Then we have from \eqref{eq:6.1} that
 $$
 b_j^{(w)} - \frac{1}{3}= \frac{1}{2} \frac{c_j^{(w)}}{c_0^{(w)}+c_1^{(w)}+c_2^{(w)}} - \frac{1}{6}
 $$
and so
\begin{equation*}
\begin{split}
\sum \left(b_j^{(w)}-\frac{1}{3}\right)^2&= \sum \left(\frac{1}{2} \frac{c_j^{(w)}}{c_0^{(w)}+c_1^{(w)}+c_2^{(w)}} - \frac{1}{6} \right)^2\\
&= \frac{1}{12}+ \frac{1}{4}\sum \frac{(c_j^{(w)})^2}{(c_0^{(w)}+c_1^{(w)}+c_2^{(w)})^2}-\frac{1}{6}\sum \frac{c_j^{(w)}}{c_0^{(w)}+c_1^{(w)}+c_2^{(w)}}\\
&= \frac{1}{4}\sum \frac{(c_j^{(w)})^2}{(c_0^{(w)}+c_1^{(w)}+c_2^{(w)})^2}- \frac{1}{12}.
\end{split}
\end{equation*}

So we need to show that
$$
\frac{1}{4}\sum \frac{(c_j^{(w)})^2}{(c_0^{(w)}+c_1^{(w)}+c_2^{(w)})^2} \leq \frac{1}{4}
$$
for which it suffices to show that
$$
c_0^{(w)} c_1^{(w)}+c_1^{(w)} c_2^{(w)}+c_0^{(w)} c_2^{(w)}>0
$$
which we will prove by induction on the length of $w$. When $|w|=1$, $c_0^{(w)} c_1^{(w)}+c_1^{(w)} c_2^{(w)}+c_0^{(w)} c_2^{(w)}= \frac{13}{15} \frac{1}{15}+\frac{13}{15} \frac{1}{15}+\frac{1}{15} \frac{1}{15}>0$.

Now suppose it is true for $w$ with $|w|=m$. We see that

\begin{equation*}
\begin{split}
 \left(
\begin{array}{ccc}
c_0^{(w0)}& c_1^{(w0)}& c_2^{(w0)}
\end{array}
 \right)&=
  \left(
\begin{array}{ccc}
c_0^{(w)}& c_1^{(w)}& c_2^{(w)}
\end{array}
 \right)M_0\\
 &= \frac{1}{15} \left(
 \begin{array}{ccc}
 9c_0^{(w)}+2c_1^{(w)}+2c_2^{(w)}&2c_1^{(w)}-c_2^{(w)} &2c_2^{(w)}-c_1^{(w)}
 \end{array}
 \right)
\end{split}
\end{equation*}

and so

\begin{equation*}
\begin{split}
15^2(c_0^{(w0)} c_1^{(w0)}+c_1^{(w0)} c_2^{(w0)}+c_0^{(w0)} c_2^{(w0)})
&= 9 c_0^{(w)}(c_1^{(w)}+c_2^{(w)})+ 2(c_1^{(w)})^2+2(c_2^{(w)})^2 \\
&+4c_1^{(w)} c_2^{(w)} +5c_1^{(w)} c_2^{(w)} -2(c_1^{(w)})^2-2(c_2^{(w)})^2\\
&=9(c_0^{(w)} c_1^{(w)}+c_1^{(w)} c_2^{(w)}+c_0^{(w)} c_2^{(w)})\\
&> 0.
\end{split}
\end{equation*}

As for the sharpness of the inequality, consider a word $w$ of length $m$ consisting of only the letter $0$. We have that
$$b_0^{(w)}=\frac{2}{3}-\frac{2}{3(3^{2m}+1)}, b_1^{(w)}=  b_2^{(w)}=\frac{1}{3(3^{2m}+1)}+\frac{1}{6}$$
and so as $m \to \infty$ we have $\lim_m \sum(b_j^{(w)}-\frac{1}{3})^2 = \frac{1}{6}$. In fact this is true for any sequence of words approaching an infinite word.
\end{proof}
\begin{rem}
\label{Atransformation}
In Theorem \ref{th:bRelation}, we saw that from $F_w SG$ to $F_{w0} SG$, the properties of $b_j^{(w)}$ were translated to the properties of the rational transformation $A_0$
$$
\left(\begin{array}{c}a_0\\ a_1\\ a_2 \end{array} \right) \mapsto
\left(\begin{array}{c}\frac{9a_0+2a_1+2a_2}{9a_0+3a_1+3a_2}\\ \frac{2a_1-a_2}{9a_0+3a_1+3a_2}\\ \frac{2a_2-a_1}{9a_0+3a_1+3a_2}\end{array}\right)$$
and we have $A_1$ and $A_2$ for the maps from $F_w SG$ to $F_{w1}SG$ and $F_{w2}SG$. Thus it is worthwhile to look at the properties of $a_j$. First, it is easy to see that $\sum{a_j} = 1$. Also, the sum of squares of $a_j$ is strictly less than $1$. To see this, using $a_j = 2\left(b_j^{(w)}-\frac{1}{6}\right)$, we have
$$\sum{\left(\frac{1}{2}a_j-\frac{1}{6}\right)^2}<\frac{1}{6}$$
so
$$\sum{\left(a_j-\frac{1}{3} \right)^2}<\frac{2}{3}.$$
Since $\sum{a_j}=1$, expanding the above gives
$$\sum{a_j^2}<1.$$
Finally, in terms of the ratios of Kusuoka measure, we have
$$\frac{\nu(F_{wj}SG)}{\nu(F_wSG)}=\frac{2}{5}\left(a_j+\frac{1}{2}\right).$$
\end{rem}

\section{Energy Distribution}
Let $b^{(w)}=(b_0^{(w)}, b_1^{(w)}, b_2^{(w)})$. We have seen in Theorem \ref{th:SameRatio} how the vector $b^{(w)}$ relates to the splitting of
Kusuoka measure in the cell $F_w SG$ when it subdivides into three subcells at the next level. Theorem \ref{th:bRelation} gives a recursive algorithm to compute the $b^{(w)}$ vectors. In this section we investigate the distribution of the collection of all $b^{(w)}$ vectors as $w$ varies over all words of fixed length $m$. We may write
$b^{w0}=B_0(b^{w})$, $b^{(w1)}=B_1(b^{(w)})$ and $b^{(w2)}=B_2(b^{(w)})$
where the $B_j$ are the rational maps given in Theorem \ref{th:bRelation}. We
may regard $\{B_0, B_1, B_2\}$ as an iterated function system acting on the open disk described in Theorem \ref{thm:bValueCircle}.\\

\begin{figure}
  \centering
    \includegraphics[height=0.77\textwidth]{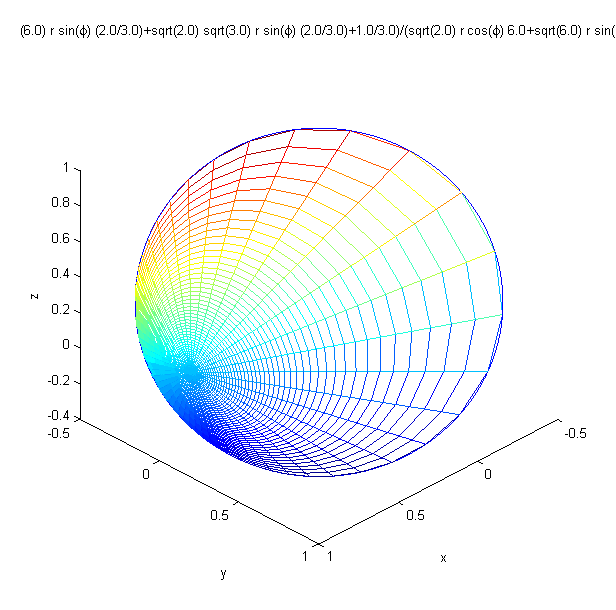}
  \caption{The map $B_0$}
  \label{fig:diskWithoutGrid}
\end{figure}

\begin{figure}
  \centering
    \includegraphics[height=0.58\textwidth]{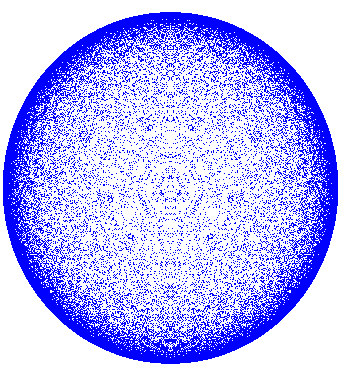}
  \caption{Points of $\{b^{(w)}:|w|=14\}$}
    \label{fig:blvl14}
\end{figure}

In Figure \ref{fig:diskWithoutGrid} we illustrate the geometric structure of $B_0$ by showing the image under $B_0$ of a uniform polar coordinate grid on the disk. $B_1$ and $B_2$ are similar, but are rotated through angles $\frac{2\pi}{3}$ and $\frac{-2\pi}{3}$.\\

To obtain $\{b^{(w)}:|w|=m\}$ we apply all $3^m$ iterates of $\{B_0, B_1, B_2\}$ to the initial vector $(\frac{1}{3}, \frac{1}{3}, \frac{1}{3})$. In Figure \ref{fig:blvl14} we show the result for $m=12$. To better understand this distribution of points in the disk we examine its angular and radial projections.\\

Introduce polar coordinates $(r,\theta)$ for the disk with $r<\frac{1}{\sqrt{6}}$(the origin corresponds to the vector $b=(\frac{1}{3}, \frac{1}{3}, \frac{1}{3})$). The angular distribution at level $m$ is $$P_m(A)=3^{-m}\#\{\text{points with }\theta\in A\}.$$
By symmetry it suffices to understand the distribution for $A\subseteq[0,\frac{2\pi}{3}]$. In Figure \ref{fig:AngDist} we show histograms (100 slices) of $P_m$ on $[0,\frac{2\pi}{3}]$ for $m=11$ and $m=13$; the values are normalized so that the average over all the values is $1$.

\begin{conj}
\label{conj1}
Let $P_m$ be defined as above. The limit of $P_m$ as $m\to \infty$ is an absolutely continuous $(B_0, B_1, B_2)$-invariant measure.
\end{conj}

Similarly, the radial distribution at level $m$ is $$Q_m(A)=3^{-m}\#\{\text{points with }r\in A\}.$$
In Figure \ref{fig:RadDist} we show histograms (300 bins) of $Q_m$ on $[0,\frac{1}{\sqrt{6}}]$, for $m=10$ and $m=14$; the values are the ratios over the total number of $b$ vectors.
\begin{conj}
\label{conj2}
Let $Q_m$ be defined as above. The limit
of $Q_m$ as $m\to \infty$ is the delta measure on the boundary $r=\frac{1}{\sqrt{6}}$.
\end{conj}

We can say more about the mappings. Recall that we have defined
$$B_0(x,y,z)=\left( \frac{9x}{13x+y+z}, \frac{2x+2y-z}{13x+y+z}, \frac{2x-y+2z}{13x+y+z}\right);$$
in Theorem \ref{th:bRelation}. For each word $w$, we have $B_0(b^{(w)})=b^{(w0)}$.
\begin{figure}
    \begin{minipage}[t]{0.5\linewidth}
        \centering
        \includegraphics[width=2.41in]{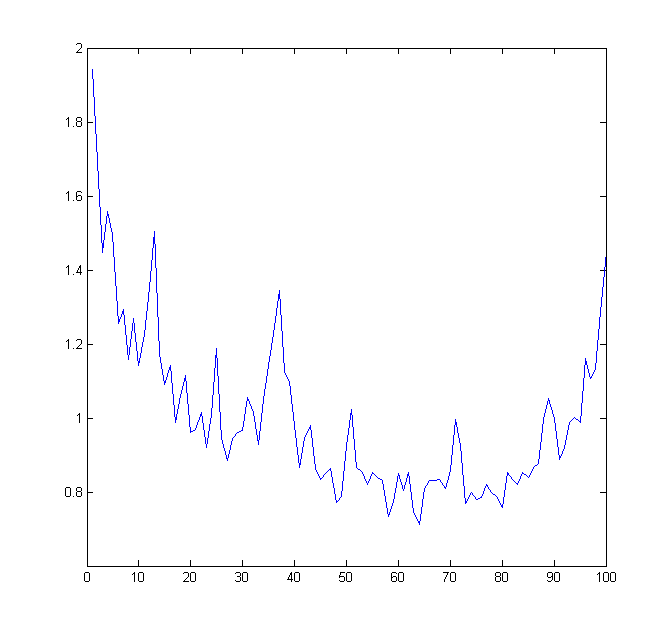}
        \subcaption{Level 11}
        \label{lvl11slice100}
    \end{minipage}%
    \begin{minipage}[t]{0.5\linewidth}
        \centering
        \includegraphics[width=2.45in]{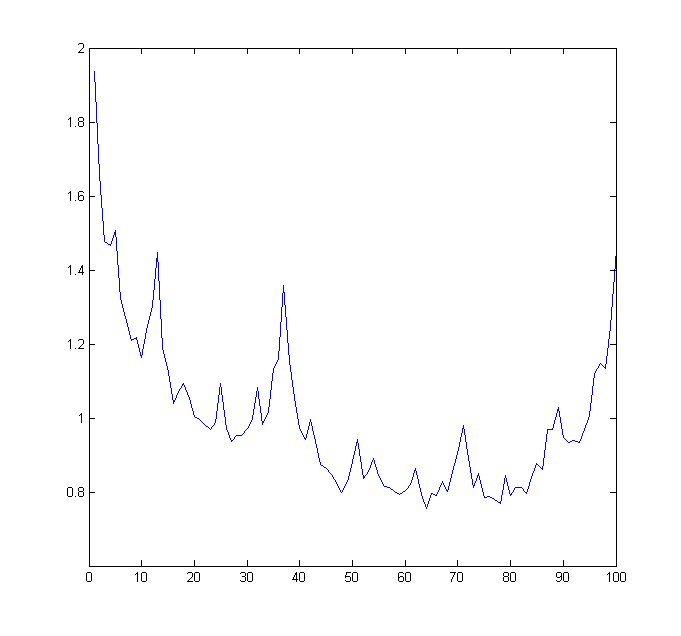}
        \subcaption{Level 13}
        \label{lvl13slice100}
    \end{minipage}
    \caption{Angular Distribution}\label{fig:AngDist}
\end{figure}

\begin{figure}
    \begin{minipage}[t]{0.5\linewidth}
        \centering
        \includegraphics[width = 2.2 in]{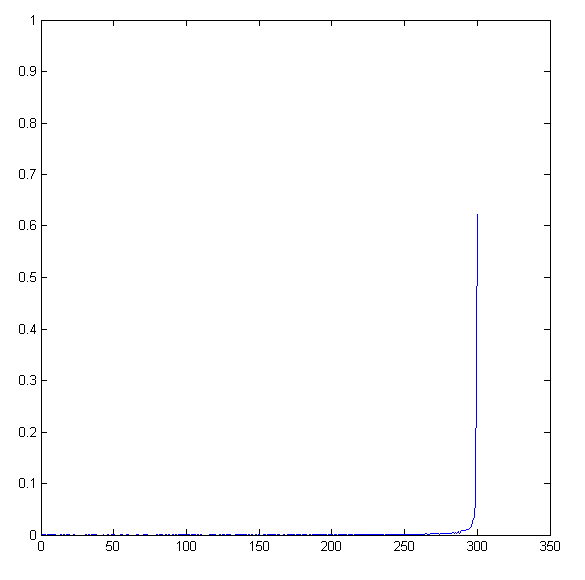}
        \subcaption{Level 10}
        \label{lvl10rad}
    \end{minipage}%
    \begin{minipage}[t]{0.5\linewidth}
        \centering
        \includegraphics[width = 2.45 in]{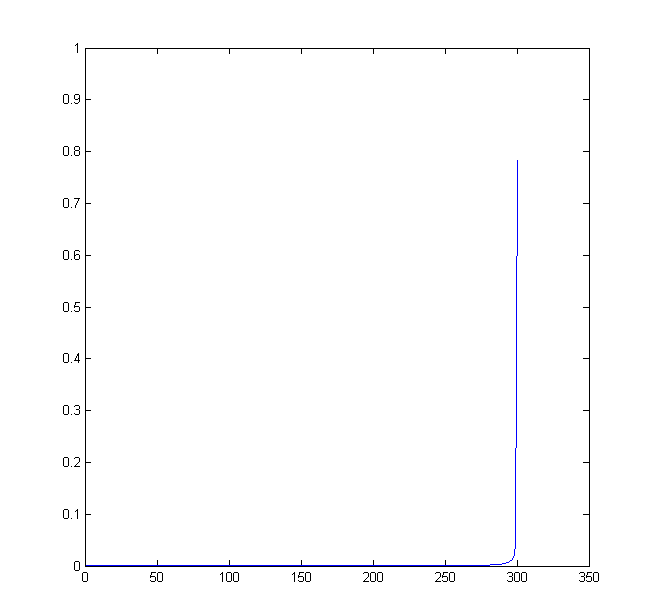}
        \subcaption{Level 14}
        \label{lvl14rad}
    \end{minipage}
    \caption{Radial Distribution}\label{fig:RadDist}
\end{figure}
We first introduce the ``polar coordinates" on the disk centered at $(\frac{1}{3}, \frac{1}{3}, \frac{1}{3})$ on the plane $\{x+y+z=1\}$. We know that $(\cos(\theta+\frac{\pi}{6}), \sin(\theta+\frac{\pi}{6}), 0)$ is a parametrization of the unit circle lying on the $x-y$ plane. Put
$$A = \left(
\begin{array}{ccc}
\frac{1}{\sqrt{2}}&\frac{1}{\sqrt{6}}&\frac{1}{\sqrt{3}}\\[1pt]
\frac{-1}{\sqrt{2}}&\frac{1}{\sqrt{6}}&\frac{1}{\sqrt{3}}\\[1pt]
0&\frac{-2}{\sqrt{6}}&\frac{1}{\sqrt{3}}
\end{array}
\right)$$
which maps the unit circle lying on the $x-y$ plane to the unit circle lying on the plane $\{x+y+z=0\}$. It follows that a parametrization of the circle of radius $r$ centered at $(\frac{1}{3}, \frac{1}{3}, \frac{1}{3})$ on the plane $\{x+y+z=1\}$ is
\begin{equation}
\label{eq:parametrizaton}
\left(
\begin{array}{c}
\frac{1}{3}\\[1pt]
\frac{1}{3}\\[1pt]
\frac{1}{3}
\end{array}
\right)+ A
\left(\begin{array}{c}
r\cos(\theta+\frac{\pi}{6})\\
r\sin(\theta+\frac{\pi}{6})\\
0
\end{array}
\right)=
\left(
\begin{array}{c}
\frac{1}{3}\\[1pt]
\frac{1}{3}\\[1pt]
\frac{1}{3}
\end{array}
\right)+r \left(
\begin{array}{c}
\frac{2}{\sqrt{6}}\cos\theta\\[3pt] \frac{1}{\sqrt{2}}\sin\theta-\frac{1}{\sqrt{6}}\cos\theta\\[3pt] \frac{-1}{\sqrt{2}}\sin\theta-\frac{1}{\sqrt{6}}\cos\theta \end{array} \right). \end{equation}
With this parametrization, the image of radius $r$ under $B_0$ is
\begin{equation}
\label{eq:RthetaMap}
\left(
\begin{array}{c} \frac{1}{3}\\[1pt] \frac{1}{3}\\[1pt] \frac{1}{3}
\end{array}
\right)
+
\frac{1}{6}\left(
\begin{array}{c}
\frac{10\sqrt{6}r\cos\theta+8}{4\sqrt{6}r\cos\theta+5}\\[3pt] \frac{9\sqrt{2}r\sin\theta-5\sqrt{6}r\cos\theta-4}{4\sqrt{6}r\cos\theta+5}\\[3pt] \frac{-9\sqrt{2}r\sin\theta-5\sqrt{6}r\cos\theta-4}{4\sqrt{6}r\cos\theta+5} \end{array} \right).\end{equation} Thus if we take $r=\frac{1}{\sqrt{6}}$, we see that the image of the boundary circle is
\begin{equation}
\label{eq:MapedBoundary}
\left(\begin{array}{c}
\frac{1}{3}\\[1pt]
\frac{1}{3}\\[1pt]
\frac{1}{3}
\end{array}
\right) +
\frac{1}{6}\left( \begin{array}{c}
\frac{10\cos\theta+8}{4\cos\theta+5}\\[3pt]
\frac{3\sqrt{3}\sin\theta-5\cos\theta-4}{4\cos\theta+5}\\[3pt]
\frac{-3\sqrt{3}\sin\theta-5\cos\theta-4}{4\cos\theta+5}
\end{array}
\right).
\end{equation}

Each $B_j$ maps the disk onto itself injectively. We first show that the mappings map the disk into itself. By \eqref{eq:parametrizaton} and \eqref{eq:RthetaMap}, we need to see, by considering $B_0$, whether to each pair $(r,\theta)$, there exists a unique solution $(\gamma, \alpha)$ satisfying
$$
\left(\begin{array}{c}
\frac{10\sqrt{6}r\cos\theta+8}{4\sqrt{6}r\cos\theta+5}\\[3pt]
\frac{9\sqrt{2}r\sin\theta-5\sqrt{6}r\cos\theta-4}{4\sqrt{6}r\cos\theta+5}\\[3pt]
\frac{-9\sqrt{2}r\sin\theta-5\sqrt{6}r\cos\theta-4}{4\sqrt{6}r\cos\theta+5}
\end{array}\right)=
\left(\begin{array}{c}
2\gamma\sqrt{6}\cos\alpha\\[3pt]
3\gamma\sqrt{2}\sin\alpha-\gamma\sqrt{6}\cos\theta\\[3pt]
-3\gamma\sqrt{2}\sin\alpha-\gamma\sqrt{6}\cos\theta
\end{array}\right).
$$
The above reduces to
\begin{equation}
\label{eq:PtToPtOnTheDisk}
\begin{split}
&\gamma\cos\alpha=\frac{1}{\sqrt{6}}\frac{5\sqrt{6}r\cos\theta+4}{4\sqrt{6}r\cos\theta+5}\\
&\gamma\sin\alpha=\frac{3r\sin\theta}{4\sqrt{6}r\cos\theta+5}.
\end{split}
\end{equation}
In particular,
\begin{equation}
\label{eq:Gamma-R}
\gamma^2-\frac{1}{6}=\frac{9\left(r^2-\frac{1}{6}\right)}{(4\sqrt{6}r\cos\theta+5)^2}
\end{equation}
so if $r\leq\frac{1}{\sqrt{6}}$, then $\gamma \leq \frac{1}{\sqrt{6}}$ and there is a solution $(\gamma, \alpha)$ satisfying \eqref{eq:PtToPtOnTheDisk}. But by $\eqref{eq:PtToPtOnTheDisk}$, we see that actually $(r,\theta)$ is also completely determined by $(\gamma, \alpha)$; it follows that the $B_j$ maps the disk onto itself injectively.\\

When we restrict the map on the boundary circle, we see from \eqref{eq:Gamma-R} that the circle is mapped onto itself and using the notations as before,
\begin{equation}
\label{eq:PolarBoundaryMap}
\begin{split}
&\cos\alpha=\frac{5\cos\theta+4}{4\cos\theta+5}\\
&\sin\alpha=\frac{3\sin\theta}{4\cos\theta+5}.
\end{split}
\end{equation}
These show that $\alpha$ corresponds to exactly one $\theta$; the map $g_j$, referring to the restriction of $B_j$ to the boundary circle, given by $g_j(\theta)=\alpha$ is well-defined. Differentiating both equations in \eqref{eq:PolarBoundaryMap}, we have $\frac{d g_0(\theta)}{d\theta}=\frac{3}{4\cos\theta+5}$. It follows that
$$g_0(\theta)=\int_0^\theta \frac{3dt}{4\cos t+5}=2\tan^{-1}\left(\frac{1}{3}\tan\frac{\theta}{2}\right)$$
and
\begin{equation*}
g_0^{-1}(\alpha)=2\tan^{-1}\left(3\tan\frac{\alpha}{2}\right).
\end{equation*}
Since $g_1$ and $g_2$ differ from $g_0$ a rotation, we have
\begin{equation*}
\begin{split}
g_1(\theta)=2\tan^{-1}\left(\frac{1}{3}\tan\left(\frac{\theta}{2}-\frac{\pi}{3}\right)\right)+\frac{2\pi}{3}\\
g_1^{-1}(\alpha)=2\tan^{-1}\left(3\tan\left(\frac{\alpha}{2}-\frac{\pi}{3}\right)\right)+\frac{2\pi}{3}\\
g_2(\theta)=2\tan^{-1}\left(\frac{1}{3}\tan\left(\frac{\theta}{2}+\frac{\pi}{3}\right)\right)-\frac{2\pi}{3}\\
g_2^{-1}(\alpha)=2\tan^{-1}\left(3\tan\left(\frac{\alpha}{2}+\frac{\pi}{3}\right)\right)-\frac{2\pi}{3}.
\end{split}
\end{equation*}
We have the following theorem.
\begin{thm}
\label{thm:BCircleAngle}
The mappings $B_j$ $(j=0,1,2)$ map the disk $\overline{U}$ centered at $(\frac{1}{3}, \frac{1}{3}, \frac{1}{3})$ of radius $\frac{1}{\sqrt{6}}$ onto itself injectively. Further,\\
\begin{itemize}
\item[(i)] the mappings map the boundary circle $C$ centered at $(\frac{1}{3}, \frac{1}{3}, \frac{1}{3})$ of radius $\frac{1}{\sqrt{6}}$ onto itself; and
\item[(ii)] If $\Psi$ is the function in \eqref{eq:parametrizaton} which associates each $(\frac{1}{\sqrt{6}},\theta)\in\R$ to $C$, there are differentiable maps $g_j$ satisfying $B_j(\Psi(\frac{1}{\sqrt{6}},\theta))=\Psi (\frac{1}{\sqrt{6}}, g_j(\theta))$ given by
    \begin{equation*}
    \begin{split}
    g_0(\theta)&=2\tan^{-1}\left(\frac{1}{3}\tan\frac{\theta}{2}\right)\\
    g_1(\theta)&=2\tan^{-1}\left(\frac{1}{3}\tan\left(\frac{\theta}{2}-\frac{\pi}{3}\right)\right)+\frac{2\pi}{3}\\
    g_2(\theta)&=2\tan^{-1}\left(\frac{1}{3}\tan\left(\frac{\theta}{2}+\frac{\pi}{3}\right)\right)-\frac{2\pi}{3}
    \end{split}
    \end{equation*}
\item[(iii)] Each $B_j$ has exactly two fixed points on the circle. The fixed points are, in ``polar coordinates", $\Psi(\frac{1}{\sqrt{6}},0)$, $\Psi(\frac{1}{\sqrt{6}},\pi)$ for $B_0$, $\Psi(\frac{1}{\sqrt{6}},\frac{2\pi}{3})$, $\Psi(\frac{1}{\sqrt{6}},\frac{-\pi}{3})$ for $B_1$ and $\Psi(\frac{1}{\sqrt{6}},\frac{-2\pi}{3})$, $\Psi(\frac{1}{\sqrt{6}},\frac{\pi}{3})$ for $B_2$.
\end{itemize}
\end{thm}
\begin{proof}
We have already proved (i) and (ii). For (iii), we only need to prove for $B_0$ and the other two follow. By (ii), $\Psi(\frac{1}{\sqrt{6}},0)$, $\Psi(\frac{1}{\sqrt{6}},\pi)$ are the only fixed points on the boundary circle. So it suffices to show that there is no fixed point in the open disk. Suppose, on the contrary, we have such a fixed point, say $\Psi(r, \theta)$. Then we must have
$$\frac{1}{6}\frac{10\sqrt{6}r\cos\theta+8}{4\sqrt{6}r\cos\theta+5}=\frac{2r}{\sqrt{6}}\cos\theta$$
which shows, since $r<\frac{1}{\sqrt{6}}$, $$\cos^2\theta=\frac{1}{6r^2}>1$$
but this contradicts the fact that $|\cos\theta|\leq 1$.
\end{proof}
By (ii), We can regard $B_j$ and $g_j$ on the boundary circle as the same map under the parametrization \eqref{eq:parametrizaton}.\\

If we directly view the maps $B_j$ as maps on the unit disk, we have an easy description on how they act on the unit disk.
\begin{thm}
\label{thm:GeoOfMaps}
Using the notation $(r,\theta)$ and $(\gamma,\alpha)$ as before. Let $(x, y)=\\(\sqrt{6}r\cos\theta, \sqrt{6}r\sin\theta)$ and $(x',y')=(\sqrt{6}\gamma\cos\theta, \sqrt{6}\gamma\sin\theta)$ be on the unit disk. Then $B_0$ maps every vertical line to a vertical line and every point $(x,y)$ with $x>-\frac{1}{2}$ closer the boundary circle. In other words, if we denote $T$ to be the triangle with vertices $(1,0)$, $(-\frac{1}{2}, \frac{\sqrt{3}}{2})$, $(-\frac{1}{2}, \frac{-\sqrt{3}}{2})$, then every point in $T$ is mapped closer to the boundary by all of the three mappings $B_j$ and every point outside $T$ is mapped closer to the boundary by two of the mappings $B_j$.
\end{thm}
\begin{proof}
The map $B_0$ on the unit disk is of the form
$$B_0(x,y)=(x',y')=\left(\frac{5x+4}{4x+5}, \frac{3y}{4x+5}\right)$$
Fix a constant $c\in [-1, 1]$. If $x=c$, then $x'=\frac{5c+4}{4c+5}$; so $B_0'$ maps every vertical line to a vertical line. Also, by \eqref{eq:Gamma-R}, we have
$$\gamma^2-\frac{1}{6}=\frac{9}{(4x+5)^2}\left(r^2-\frac{1}{6}\right);$$
thus if $x>-\frac{1}{2}$, $\frac{9}{(4x+5)^2}<1$ and hence
$$\gamma^2-\frac{1}{6}<\frac{9}{(4x+5)^2}\left(r^2-\frac{1}{6}\right).$$
The last assertion comes from the symmetry of $B_j$, since all of these mappings differ by a rotation of $\frac{2\pi}{3}$ and $\frac{-2\pi}{3}$ only.
\end{proof}

By Theorem \ref{thm:BCircleAngle}, we can say more about the measure mentioned in Conjecture \ref{conj1}. We call the measure $d\lambda(t)=f(t)dt$ for some $f$, $\lambda[a,b]=\lambda(B_j[a,b])$ for all $j=0,1,2$. Recall that it is $(B_0, B_1, B_2)$-invariant, hence
\begin{equation*}
\begin{split}
\int_a^b f &= \int_{g_i^{-1}[a,b]} f(t) dt\\
 &= \int_a^b f(g_i^{-1}(t))(g_i^{-1})'(t)dt\\
 &=\int_a^b \frac{f(g_i^{-1}(t))}{g_i'(g_i^{-1}(t))}dt\\
 &=\frac{1}{3}\sum_{j=0}^2 \int_a^b \frac{f(g_j^{-1}(t))}{g_j'(g_j^{-1}(t))}dt.
\end{split}
\end{equation*}
It follows that the function $f$ can be characterized by one satisfying the relation
$$f(t)=\frac{1}{3}\sum_{j=0}^2 \frac{f(g_j^{-1}(t))}{g_j'(g_j^{-1}(t))}.$$
If we consider $g_j$ separately, we have
\begin{equation*}
\begin{split}
\frac{1}{g_0'(g_0^{-1}(t))}&=\frac{4}{3}\cos\left(2\tan^{-1}\left(3\tan\frac{t}{2}\right)\right)+\frac{5}{3}=\frac{3}{5-4\cos t};\\
\frac{1}{g_1'(g_1^{-1}(t))}&=\frac{3}{5-4\cos \left(t-\frac{2\pi}{3}\right)};\\
\frac{1}{g_2'(g_2^{-1}(t))}&=\frac{3}{5-4\cos \left(t+\frac{2\pi}{3}\right)}.
\end{split}
\end{equation*}

Finally, we give further experimental evidence for the conjectures. Conjecture \ref{conj2} states that the radial distribution is the delta measure at $r=\frac{1}{\sqrt{6}}$; the angular distribution on the whole disk should be the angular distribution on the boundary circle only. Pick three points $(0,1)$, $(-1, 0)$, $(\frac{1}{\sqrt{2}}, -\frac{1}{\sqrt{2}})$ from the three arcs mentioned in Theorem \ref{thm:GeoOfMaps} respectively and iterate them $14$ times by all of the three maps. We cut the whole circle into 800 bins and plot the histogram for the number of points (normalized so that the average over all the values is $1$) in each bin for bins covering one sixth of the circle in Figure \ref{1sixth800slicesNormalized}. This should be compared to Figure \ref{lvl13slice100} which gives the level 13 histogram.

\begin{figure}
        \centering
        \includegraphics[height = 3 in]{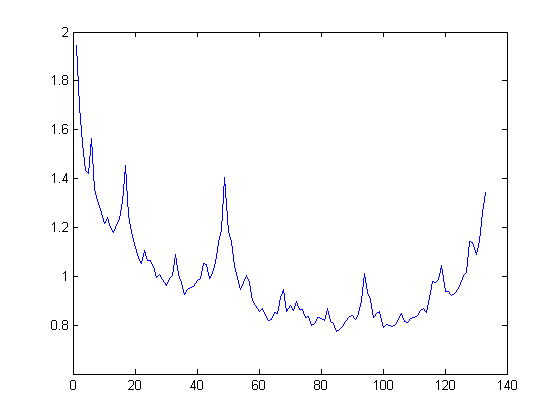}
        \caption{Angular Distribution on Boundary}
        \label{1sixth800slicesNormalized}
\end{figure}

\section{Acknowledgement}
We are grateful to the referee for many useful suggestions.

\end{document}